\documentclass[oneside,english]{amsart}
\usepackage[T1]{fontenc}
\usepackage[latin9]{inputenc}
\usepackage{amsthm}
\usepackage{amstext}
\usepackage{amssymb}
\usepackage{cite}
\usepackage{esint}
\usepackage{float}
\usepackage{graphicx}
\usepackage{multirow}
\usepackage{subfig}
\floatstyle{plaintop}
\restylefloat{table}

 \makeatletter

\makeatletter

\numberwithin{equation}{section}
\numberwithin{figure}{section}
\theoremstyle{plain}
\newtheorem{thm}{\protect\theoremname}[section]
  \theoremstyle{definition}
  \newtheorem{defn}[thm]{\protect\definitionname}
  \theoremstyle{plain}
  \newtheorem{prop}[thm]{\protect\propositionname}
  \theoremstyle{remark}
  \newtheorem{rem}{\protect\remarkname}[section]
  \theoremstyle{plain}
  
  \theoremstyle{plain}
  \newtheorem{cor}[thm]{\protect\corollaryname}
  \theoremstyle{plain}
  \newtheorem{lem}[thm]{\protect\lemmaname}

\makeatletter
\@addtoreset{thm}{section}
\makeatother

\makeatother

\usepackage{babel}
  \providecommand{\corollaryname}{Corollary}
  \providecommand{\definitionname}{Definition}
  \providecommand{\factname}{Fact}
  \providecommand{\lemmaname}{Lemma}
  \providecommand{\propositionname}{Proposition}
  \providecommand{\remarkname}{Remark}
  
\providecommand{\theoremname}{Theorem}

\begin{document}

\title[Description of Stability for Linear Time-Invariant Systems Based on Curvature]
{Description of Stability for Linear Time-Invariant Systems Based on the First Curvature}

\author[Y. Wang]{Yuxin Wang $^1$}
\author[H. Sun]{Huafei Sun $^{1,*}$}
\author[S. Huang]{Shoudong Huang $^2$}
\author[Y. Song]{Yang Song $^1$}

\thanks{This subject is supported by the National Natural Science Foundations of China (No. 61179031.)}
\thanks{$^1$ School of Mathematics and Statistics, Beijing Institute of Technology, Beijing 100081, P.~R.~China}
\thanks{$^2$ Faculty of Engineering and Information Technology, University of Technology Sydney, Australia}
\thanks{$^*$ Huafei Sun is the corresponding author}
\thanks{E-mail: wangyuxin@bit.edu.cn,~ huafeisun@bit.edu.cn,~ sdhuang@eng.uts.edu.au,~ frank230316@126.com}

\begin{abstract}
This paper focuses on using the first curvature $\kappa(t)$ of trajectory to describe the stability of linear time-invariant system.
We extend the results for two and three-dimensional systems [Y.~Wang, H.~Sun, Y.~Song et al., arXiv:1808.00290] to $n$-dimensional systems.
We prove that for a system $\dot{r}(t)=Ar(t)$,
(i) if there exists a measurable set whose Lebesgue measure is greater than zero,
such that for all initial values in this set,
$\lim\limits_{t\to+\infty}\kappa(t)\neq0$ or $\lim\limits_{t\to+\infty}\kappa(t)$ does not exist,
then the zero solution of the system is stable;
(ii) if the matrix $A$ is invertible,
and there exists a measurable set whose Lebesgue measure is greater than zero,
such that for all initial values in this set,
$\lim\limits_{t\to+\infty}\kappa(t)=+\infty$,
then the zero solution of the system is asymptotically stable.
\end{abstract}



\maketitle

\section{Introduction}

It is well known that stability is an important subject in the control theory,
and curvature is the core concept of differential geometry.
We wish to establish the relationship between the curvatures of state trajectories and the stability of linear systems.
In fact, in \cite{Wang} the authors gave the description of stability
for two and three-dimensional linear time-invariant systems $\dot{r}(t)=Ar(t)$ based on the curvature and torsion of curve $r(t)$.

In this paper, we focus on the higher dimensional systems and give them a geometric description for the stability. To achieve this goal, the definition of the higher curvatures of curves in $\mathbb{R}^n$ by Gluck \cite{Gluck} is used,
where the first and second curvature are the generalization of curvature and torsion of curves in $\mathbb{R}^3$, respectively.
We will develop the methods arised in \cite{Wang},
and use the first curvature to describe the stability of the zero solution of linear time-invariant system $\dot{r}(t)=Ar(t)$.

Our main results are as follows.

\begin{thm} \label{thm main}
Suppose that $\dot{r}(t)=Ar(t)$ is a linear time-invariant system, where $A$ is an $n\times n$ real matrix, $r(t)\in\mathbb{R}^n$, and $\dot{r}(t)$ is the derivative of $r(t)$.
Denote by $\kappa(t)$ the first curvature of trajectory of a solution $r(t)$. We have
\par
$(1)$ if there exists a measurable set $E_1\subseteq \mathbb{R}^n$ whose Lebesgue measure is greater than $0$, such that for all $r(0)\in E_1$, $\lim\limits_{t\to+\infty}\kappa(t)\neq0$ or $\lim\limits_{t\to+\infty}\kappa(t)$ does not exist, then the zero solution of the system is stable;
\par
$(2)$ if $A$ is invertible, and there exists a measurable set $E_2\subseteq \mathbb{R}^n$ whose Lebesgue measure is greater than $0$, such that for all $r(0)\in E_2$, $\lim\limits_{t\to+\infty}\kappa(t)=+\infty$, then the zero solution of the system is asymptotically stable.
\end{thm}

The paper is organized as follows.
In Section \ref{Section Preliminaries}, we review some basic concepts and propositions.
In Section \ref{Section Relationship}, we establish the relationship between curvatures of trajectories of two equivalent systems.
In Section \ref{Section Jordan Blocks}, we discuss four types of real Jordan blocks.
In Section \ref{Section General}, we consider the case of real Jordan canonical form, and complete the proof of Theorem \ref{thm main}.
Several examples are given in Section \ref{Section Examples}.
Finally, Section \ref{Section Conclusion} concludes the paper.

\section{Preliminaries}\label{Section Preliminaries}
In this paper, the norm $\|x\|$ denotes the Euclidean norm of $x=\left(x_1,x_2,\cdots,x_n\right)^\mathrm{T}\in\mathbb{R}^n$, namely, $\|x\|=\sqrt{\sum_{i=1}^n x_i^2}$.
Denote by $\det A$ the determinant of matrix $A$.
The eigenvalues of matrix $A$ are denoted by $\lambda_i(A)\,(i=1,2,\cdots,n)$,
and the set of eigenvalues of matrix $A$ is denoted by $\sigma(A)$.

The following concepts and results can be found in \cite{Carmo,Chen,Gluck,Horn,Lyapunov,Marsden,Perko}.

\subsection{Linear Time-Invariant Systems and Stability}

\begin{defn}[\!\!\cite{Perko}]
The system of ordinary differential equations
\begin{align}\label{system1}
\dot{r}(t)=Ar(t)
\end{align}
is called a linear time-invariant system, where $A$ is an $n\times n$ real constant matrix, $r(t)\in\mathbb{R}^n$, and $\dot{r}(t)$ is the derivative of $r(t)$.
\end{defn}

\begin{prop}[\!\!\cite{Perko}]\label{ODE}
Let $A$ be an $n\times n$ real matrix. Then for a given $r_0\in\mathbb{R}^n$, the initial value problem
\begin{align}
\left\{
\begin{aligned}\label{system}
\dot{r}(t)&=Ar(t),\\
r(0)&=r_0
\end{aligned}
\right.
\end{align}
has a unique solution given by
\begin{align}\label{align solution}
r(t)=\mathrm{e}^{tA}r_0.
\end{align}
\par
The curve $r(t)$ is called the trajectory of system (\ref{system}) with the initial value $r_0\in\mathbb{R}^n$.
\end{prop}

\begin{defn}[\!\!\cite{Chen,Marsden}]
The solution $r(t)\equiv0$ of differential equations (\ref{system1}) is called the zero solution of the linear time-invariant system.
If for every constant $\varepsilon >0$,
there exists a $\delta=\delta(\varepsilon)>0$,
such that $\|r(0)\|<\delta $ implies that $\|r(t)\|<\varepsilon$ for all $t\in [0,+\infty)$,
where $r(t)=\mathrm{e}^{tA}r(0)$ is a solution of (\ref{system1}),
and $r(0)$ is the initial value of $r(t)$,
then we say that the zero solution of system (\ref{system1}) is stable.
If the zero solution is not stable, then we say that it is unstable.
\par
Suppose that the zero solution of system (\ref{system1}) is stable,
and there exists a $\delta_1\,(0<\delta_1\leqslant\delta)$,
such that $\|r(0)\|<\delta_1$ implies that $\lim\limits_{t\to+\infty}r(t)=0$,
then we say that the zero solution of system (\ref{system1}) is asymptotically stable.
\end{defn}

\begin{prop}[\!\!\cite{Chen}]\label{prop asy.stable}
The zero solution of system (\ref{system1}) is stable if and only if all eigenvalues of matrix $A$ have nonpositive real parts, namely,
\begin{align*}
\mathrm{Re}\{\lambda_i(A)\}\leqslant 0\quad (i=1,2,\cdots,n),
\end{align*}
and the eigenvalues with zero real parts correspond only to the simple elementary factors of matrix $A$.
\par
The zero solution of system (\ref{system1}) is asymptotically stable if and only if all eigenvalues of matrix $A$ have negative real part, namely,
\begin{align*}
\mathrm{Re}\{\lambda_i(A)\}<0\quad (i=1,2,\cdots,n).
\end{align*}
\end{prop}

\begin{prop}[\!\!\cite{Chen}]
Suppose that $A$ and $B$ are two $n\times n$ real matrices, and $A$ is similar to $B$, namely, there exists an $n\times n$ real invertible matrix $P$, such that $A=P^{-1}BP$. For system (\ref{system1}), let $v(t)=Pr(t)$. Then the system after the transformation becomes
\begin{align}\label{system2}
\dot{v}(t)=Bv(t).
\end{align}
System (\ref{system2}) is said to be equivalent to system (\ref{system1}), and $v(t)=Pr(t)$ is called an equivalence transformation.
\end{prop}

\begin{prop}[\!\!\cite{Chen}]
Let $A$ and $B$ be two $n\times n$ real matrices, and $A$ is similar to $B$. Then the zero solution of the system $\dot{r}(t)=Ar(t)$ is (asymptotically) stable if and only if the zero solution of the system $\dot{v}(t)=Bv(t)$ is (asymptotically) stable.
\end{prop}

\subsection{Curvatures of Curves in $\mathbb{R}^n$}

\begin{defn}[\!\!\cite{Carmo}]
Let $r:[0, +\infty)\to\mathbb{R}^3$ be a smooth curve. The functions
\begin{align*}
\kappa(t)=\frac{\left\|\dot{r}(t)\times\ddot{r}(t)\right\|}{\left\|\dot{r}(t)\right\|^3}, \quad
\tau(t)=\frac{\left(\dot{r}(t),\ddot{r}(t),\dddot{r}(t)\right)}{\left\|\dot{r}(t)\times\ddot{r}(t)\right\|^2}
\end{align*}
are called the curvature and torsion of curve $r(t)$, respectively.
\end{defn}
\par
Gluck \cite{Gluck} gave the definition of higher curvatures of curves in $\mathbb{R}^n$,
which is a generalization of curvature and torsion.
Here we briefly review the results of \cite{Gluck}.

Let $r:[0, +\infty)\to\mathbb{R}^n$ be a smooth curve, and $\dot r(t)\neq0$ for all $t\in[0, +\infty)$.
Suppose that for each $t\in[0, +\infty)$, the vectors
\begin{align}\label{align vectors}
\dot r(t), \ddot r(t), \cdots, r^{(m)}(t)
\end{align}
are linearly independent. Applying the Gram-Schmidt orthonormalization process to (\ref{align vectors}),
we obtain the orthogonal vectors
\begin{align*}
E_1(t), E_2(t), \cdots, E_m(t),
\end{align*}
and an orthonormal set whose elements are
\begin{align*}
H_1(t), H_2(t), \cdots, H_m(t),
\end{align*}
where
\begin{align}\label{align Gram-Schmidt}
& E_1(t)=\dot r(t),
\\& E_2(t)=\ddot r(t)-\frac{ \left\langle \ddot r(t), E_{1}(t) \right\rangle }{ \left\langle E_{1}(t), E_{1}(t) \right\rangle } E_{1}(t),
\nonumber\\& \cdots\cdots
\nonumber\\& E_m(t)=r^{(m)}(t)-\sum_{i=1}^{m-1}\frac{ \left\langle r^{(m)}(t), E_{i}(t) \right\rangle }{ \left\langle E_{i}(t), E_{i}(t) \right\rangle } E_{i}(t).
\nonumber
\end{align}
and
\begin{align*}
H_i(t)=\frac{ E_i(t) }{ \left\| E_i(t) \right\| } \quad
(i=1,2,\cdots,m).
\end{align*}
Furthermore, we have
\begin{align*}
E_{m+1}(t)=r^{(m+1)}(t)-\sum_{i=1}^{m}\frac{ \left\langle r^{(m+1)}(t), E_{i}(t) \right\rangle }{ \left\langle E_{i}(t), E_{i}(t) \right\rangle } E_{i}(t).
\end{align*}
Nevertheless, if $E_{m+1}(t)=0$, then we cannot form $H_{m+1}(t)$.

\begin{defn}[\!\!\cite{Gluck}]
Let $r(s)$ be a smooth curve in $\mathbb{R}^n$, where $s$ is the arc length parameter, namely, $\left\| \dot r(s) \right\|\equiv 1$ for all $s\in[0, +\infty)$.
Suppose that
$\dot r(s), \ddot r(s), \cdots, r^{(m)}(s)$
are linearly independent,
then we have
\begin{align*}
\begin{pmatrix}
\dot H_1(s)\\[0.7em] \dot H_2(s)\\[0.7em] \vdots\\[0.7em] \dot H_{m-1}(s)
\end{pmatrix}
=
\begin{pmatrix}
  0&\kappa_1(s)&&&\\[1em]
  -\kappa_1(s)&0&\kappa_2(s)&&\\[0.6em]
  &\kappa_2(s)&0&\ddots&\\[0.6em]
  &&\ddots&\ddots&\kappa_{m-2}(s)\\[1em]
  &&&-\kappa_{m-2}(s)&0&\kappa_{m-1}(s)
\end{pmatrix}
\begin{pmatrix}
H_1(s)\\[0.7em] H_2(s)\\[0.7em] \vdots\\[0.7em] H_m(s)
\end{pmatrix},
\end{align*}
where $\kappa_1(s), \kappa_2(s), \cdots, \kappa_{m-1}(s)$ are called the first, second, $\cdots$, $(m-1)$-th curvature of the curve $r(s)$, respectively.
\par
Moreover, for $s_0\in[0, +\infty)$, if $r^{(m+1)}(s_0)$ is linearly independent of
$\dot r(s_0), \ddot r(s_0), \cdots, r^{(m)}(s_0)$,
then this will also be true in some neighborhood of $s_0$ in $[0, +\infty)$.
For $s$ in such a neighborhood, $H_{m+1}(s)$ can be defined, and we have
$\dot H_m(s)=-\kappa_{m-1}(s)H_{m-1}(s)+\kappa_m(s)H_{m+1}(s)$,
where $\kappa_m(s)$ is the $m$th curvature of the curve $r(s)$.
Conversely, if $r^{(m+1)}(s_0)$ can be represented as a linear combination of
$\dot r(s_0)$, $\ddot r(s_0)$, $\cdots$, $r^{(m)}(s_0),$
then $\kappa_m(s_0)=0$.
\end{defn}

\begin{rem}
$\kappa_i(s)>0$ for $i=1,2,\cdots,m-1$, and $\kappa_m(s)\geqslant0$.
\end{rem}

\begin{rem}
Let $r(s)$ be a smooth curve in $\mathbb{R}^3$, where $s$ is the arc length parameter. Suppose that
$
\dot r(s), \ddot r(s), \dddot r(s)
$
are linearly independent. Then we have Frenet-Serret formulas (cf. \cite{Carmo}), and $\kappa_1(s)=\kappa(s),~ \kappa_2(s)=|\tau(s)|.$
\end{rem}

\begin{rem}
Gluck \cite{Gluck} gave the formula of each curvature of curve $r(t)$ in $\mathbb{R}^n$.
In fact, suppose that $t$ is the parameter of curve $r(t)$, and $\dot r(t)\neq0$ for all $t\in[0, +\infty)$. Let
\begin{align*}
V_i(t)=\prod_{p=1}^i \left\| E_p(t) \right\|,
\end{align*}
namely, $V_i(t)$ denotes the $i$-dimensional volume of $i$-dimensional parallelotope with
vectors
$\dot{r}(t)$, $\ddot{r}(t)$, $\cdots$, $r^{(i)}(t)$
as edges, and we have a convention that $V_0(t)=1$.
Then we have the following result.
\end{rem}

\begin{prop}[\!\!\cite{Gluck}]\label{prop curvature}
The $i$th curvature of curve $r(t)$ is
\begin{align*}
\kappa_i(t)
=\frac{\left\| E_{i+1}(t) \right\|}{ \left\| E_{1}(t) \right\| \left\| E_{i}(t) \right\| }
=\frac{ V_{i+1}(t) V_{i-1}(t) }{ V_{1}(t) V_{i}^2(t) } \quad
(i=1,2,\cdots,m).
\end{align*}
\end{prop}

We can give $V_i(t)$ by the derivatives of $r(t)$ with respect to $t$, and thus obtain the expression of $\kappa_i(t)$.
In fact, write $r^{(i)}(t)=\left( r_1^{(i)}(t), r_2^{(i)}(t), \cdots, r_n^{(i)}(t) \right)^{\mathrm{T}}$,
then by Cauchy-Binet formula, we have
\begin{align}
& V_1^2(t)=\left\| \dot r(t) \right\|^2=\sum_{i=1}^n \dot r_i^2(t),
\nonumber\\
& V_2^2(t)=
\sum_{1\leqslant i<j\leqslant n}
\begin{vmatrix}
\dot{r}_{i}(t) & \ddot{r}_{i}(t) \\[0.7em]
\dot{r}_{j}(t) & \ddot{r}_{j}(t)
\end{vmatrix}^2,
\label{align volume 2}\\
& \cdots\cdots
\nonumber\\
& V_k^2(t)=
\sum_{1\leqslant i_1<i_2<\cdots <i_k\leqslant n}
\begin{vmatrix}
\dot{r}_{i_1}(t) & \ddot{r}_{i_1}(t) & \cdots & r_{i_1}^{(k)}(t) \\[0.7em]
\dot{r}_{i_2}(t) & \ddot{r}_{i_2}(t) & \cdots & r_{i_2}^{(k)}(t) \\[0.7em]
\vdots & \vdots & \ddots & \vdots \\[0.7em]
\dot{r}_{i_k}(t) & \ddot{r}_{i_k}(t) & \cdots & r_{i_k}^{(k)}(t)
\end{vmatrix}^2,
\nonumber\\
& \cdots\cdots
\nonumber
\end{align}
Hence we obtain the expression of each curvature of curve $r(t)$ in $\mathbb{R}^n$ by the coordinates of derivatives of $r(t)$.
In particular, the first curvature of $r(t)$ satisfies
\begin{align}\label{align first curvature}
\kappa_1(t)
=\frac{\left\| E_{2}(t) \right\|} { \left\| E_{1}(t) \right\|^2 }
=\frac{ V_{2}(t) }{ V_{1}^3(t) }
=\sqrt{ \frac{
\sum_{1\leqslant i<j\leqslant n}
\begin{vmatrix}
\dot{r}_{i}(t) & \ddot{r}_{i}(t) \\[0.7em]
\dot{r}_{j}(t) & \ddot{r}_{j}(t)
\end{vmatrix}^2
}
{ \left( \sum_{k=1}^n \dot r_k^2(t) \right)^3 }
}.
\end{align}
We denote $\kappa(t)$ instead of $\kappa_1(t)$, the first curvature of $r(t)$, for simplicity.

\subsection{Singular Value Decomposition}

\begin{defn}[\!\!\cite{Horn}]
Let $A$ be an $m\times n$ complex matrix with rank $r$, and $\lambda_1, \lambda_2, \cdots, \lambda_r$ the non-zero eigenvalues of $AA^\mathrm{H}$, where $A^\mathrm{H}$ denotes the conjugate transpose of $A$. Then
\begin{align*}
\delta_i=\sqrt{\lambda_i}\quad (i=1,2,\cdots,r)
\end{align*}
are called the singular values of $A$.
\end{defn}

\begin{prop}[\!\!\cite{Horn}]\label{SVD}
Let $A$ be an $m\times n$ real matrix with rank $r$, and $\delta_1\geqslant\delta_2\geqslant\cdots\geqslant\delta_r$ the singular values of $A$. Then there exists an $m\times m$ orthogonal matrix $U$ and an $n\times n$ orthogonal matrix $V$, such that
\begin{align*}
A=UDV^\mathrm{T}=U\begin{pmatrix}\Delta&0\\0&0\end{pmatrix}V^\mathrm{T},
\end{align*}
where $\Delta=\mathrm{diag} \{\delta_1, \delta_2, \cdots, \delta_r \}$.
\end{prop}

\subsection{Real Jordan Canonical Form}

\begin{prop}[\!\!\cite{Horn}]\label{prop Jordan1}
Let $A$ be an $n\times n$ real matrix. Then $A$ is similar to a block diagonal real matrix
\begin{align*}
\begin{pmatrix}
 \begin{matrix}
  C_{n_1}(a_1,b_1)&&\\[0.7ex]
  &C_{n_2}(a_2,b_2)&\\
  &&\ddots
 \end{matrix}
 &&\text{\LARGE$0$}\\
 &C_{n_p}(a_p,b_p)&\\[0.7ex]
 \text{\LARGE$0$}&&
 \begin{matrix}
  J_{n_{p+1}}(\lambda_{p+1})&&\\
  &\ddots&\\
  &&J_{n_{r}}(\lambda_r)
 \end{matrix}
\end{pmatrix},
\end{align*}
where
\par
$(1)$ for $k\in\{1,2,\cdots,p\}$, $\lambda_k=a_k+\sqrt{-1} b_k$ and $\bar\lambda_k=a_k-\sqrt{-1} b_k$
\ $(a_k,b_k\in\mathbb{R},\text{and}\ b_k>0)$ are eigenvalues, and
\begin{align*}
C_{n_k}(a_k,b_k)=
\begin{pmatrix}
\Lambda_k & I_2&&&\\[0.6em]
&\Lambda_k & I_2&&\\
&&\Lambda_k &\ddots&\\
&&&\ddots & I_2\\[0.5em]
&&&&\Lambda_k
\end{pmatrix}_{2n_k\times 2n_k},
\end{align*}
where
$\Lambda_k=
\begin{pmatrix}
 a_k&b_k\\
-b_k&a_k
\end{pmatrix},
I_2=
\begin{pmatrix}
1&0\\
0&1
\end{pmatrix};$
\par
$(2)$ for $j\in\{p+1,p+2,\cdots,r\}$, $\lambda_j\in\mathbb{R}$ is a real eigenvalue, and
\begin{align*}
J_{n_j}(\lambda_j)=
\begin{pmatrix}
\lambda_j &1&&&\\[0.6em]
&\lambda_j &1&&\\
&&\lambda_j &\ddots&\\
&&&\ddots &1\\[0.5em]
&&&&\lambda_j
\end{pmatrix}_{n_j\times n_j}.
\end{align*}
\end{prop}

\section{Relationship Between the Curvatures of Two Equivalent Systems}\label{Section Relationship}

In this section, we establish the relationship between curvatures of trajectories of two equivalent systems.

Let curve $r(t)$ be the trajectory of system (\ref{system}),
and curve $v(t)$ the trajectory of system $\dot{v}(t)=Bv(t)$, $v(0)=v_0$,
where $A=P^{-1}BP$, and $v_0=Pr_0$. Suppose that for each $t$, the vectors
\begin{align*}
\dot{r}(t), \ddot{r}(t), \cdots, r^{(m)}(t)
\end{align*}
are linearly independent. Since $v^{(i)}(t)=Pr^{(i)}(t) \,
(i=1,2,\cdots,m)$, we see that the vectors
\begin{align*}
\dot{v}(t), \ddot{v}(t), \cdots, v^{(m)}(t)
\end{align*}
are also linearly independent.
Hence, we can define curvatures $\kappa_{r,1}(t), \kappa_{r,2}(t), \cdots, \kappa_{r,m}(t)$ of curve $r(t)$,
and curvatures $\kappa_{v,1}(t), \kappa_{v,2}(t), \cdots, \kappa_{v,m}(t)$ of curve $v(t)$, respectively.

Now, we prove the following theorem.
\begin{thm}\label{thm relation}
Suppose that linear time-invariant system $\dot{r}(t)=Ar(t)$ is equivalent to system $\dot{v}(t)=Bv(t)$, where
 $A=P^{-1}BP$, and $v(t)=Pr(t)$ is the equivalence transformation.
 Let $\kappa_{r,i}(t)$ and $\kappa_{v,i}(t)$ be the $i$th $(i=1,2,\cdots,m-1)$ curvatures of trajectories $r(t)$ and $v(t)$, respectively. Then we have
\begin{align*}
\lim\limits_{t\to+\infty}\kappa_{v,i}(t)=0 &\iff \lim\limits_{t\to+\infty}\kappa_{r,i}(t)=0, \\
\lim\limits_{t\to+\infty}\kappa_{v,i}(t)=+\infty &\iff \lim\limits_{t\to+\infty}\kappa_{r,i}(t)=+\infty, \\
\kappa_{v,i}(t) ~\text{is a bounded function} &\iff \kappa_{r,i}(t) ~\text{is a bounded function}.
\end{align*}
\end{thm}

\begin{proof}
First, recall that $A=P^{-1}BP$, where $P$ is an $n\times n$ real invertible matrix. Using Proposition \ref{SVD}, we obtain a singular value decomposition of $P$, namely,
\begin{align*}
P=U\Delta V^\mathrm{T},
\end{align*}
where $U$ and $V$ are two $n\times n$ orthogonal matrices,
and $\Delta=\mathrm{diag}\{\delta_1, \delta_2, \cdots, \delta_n\} ~( \delta_1\geqslant\delta_2\geqslant\cdots\geqslant\delta_n>0 )$.
Let
\begin{align*}
a(t)=V^\mathrm{T}r(t), \quad
b(t)=\Delta a(t)=\Delta V^\mathrm{T} r(t)=U^{-1}v(t).
\end{align*}
Then we obtain two new linear time-invariant systems
\begin{align*}
\dot a(t)= Na(t), \quad
\dot b(t)= \tilde N b(t),
\end{align*}
where $N=V^\mathrm{T} A (V^\mathrm{T})^{-1}$, and $\tilde N=U^\mathrm{-1} B U$.
\par
Since $U,V$ are $n\times n$ orthogonal matrices, we obtain
\begin{align}\label{align eq. curvature}
\kappa_{r,i}(t)=\kappa_{a,i}(t),\quad
\kappa_{v,i}(t)=\kappa_{b,i}(t)\quad
(i=1,2,\cdots,m-1).
\end{align}
In fact, because $a(t)=V^\mathrm{T}r(t)$, we have $a^{(k)}(t)=V^\mathrm{T}r^{(k)}(t)\,(k=1,2,\cdots,m)$.
Let $E_{a,i}(t)$ and $E_{r,i}(t)$ denote the $i$th vector obtained by Gram-Schmidt orthogonalization (\ref{align Gram-Schmidt})
for vectors $\dot{a}(t), \ddot{a}(t), \cdots, a^{(m)}(t)$
and vectors $\dot{r}(t), \ddot{r}(t), \cdots, r^{(m)}(t)$, respectively.
Then we have $E_{a,i}(t)=V^\mathrm{T} E_{r,i}(t)\, (i=1,2,\cdots,m)$.
Using Proposition \ref{prop curvature}, the $i$th curvature of curve $a(t)$ satisfies
\begin{align*}
\kappa_{a,i}(t)
&=\frac{\left\| E_{a,i+1}(t) \right\|}{ \left\| E_{a,1}(t) \right\| \left\| E_{a,i}(t) \right\| }
=\frac{\left\| V^\mathrm{T} E_{r,i+1}(t) \right\|}{ \left\| V^\mathrm{T} E_{r,1}(t) \right\| \left\| V^\mathrm{T} E_{r,i}(t) \right\| }
\\&=\frac{\left\| E_{r,i+1}(t) \right\|}{ \left\| E_{r,1}(t) \right\| \left\| E_{r,i}(t) \right\| }
=\kappa_{r,i}(t)
\quad (i=1,2,\cdots,m-1).
\end{align*}
Similarly, we have $\kappa_{b,i}(t)=\kappa_{v,i}(t)\,(i=1,2,\cdots,m-1)$.
\par
The task is now to find the relationship between curvatures $\kappa_{a,i}(t)$ and $\kappa_{b,i}(t)$.
\par
Noting that $b(t)=\Delta a(t)$, we have $b_i(t)=\delta_i a_i(t)$,
and derivatives $b_i^{(k)}(t)=\delta_i a_i^{(k)}(t)\, (k=1,2,\cdots,m)$.
Thus, the square of the $k$-dimensional volume of $k$-dimensional parallelotope with
vectors $\dot{b}(t), \ddot{b}(t), \cdots, b^{(k)}(t)$ as edges is
\begin{align*}
V_{b,k}^2(t)
&=\sum_{1\leqslant i_1<i_2<\cdots <i_k\leqslant n}
\begin{vmatrix}
\dot{b}_{i_1}(t) & \ddot{b}_{i_1}(t) & \cdots & b_{i_1}^{(k)}(t) \\[0.7em]
\dot{b}_{i_2}(t) & \ddot{b}_{i_2}(t) & \cdots & b_{i_2}^{(k)}(t) \\[0.7em]
\vdots & \vdots & \ddots & \vdots \\[0.7em]
\dot{b}_{i_k}(t) & \ddot{b}_{i_k}(t) & \cdots & b_{i_k}^{(k)}(t)
\end{vmatrix}^2
\displaybreak[0]
\\&=\sum_{1\leqslant i_1<i_2<\cdots <i_k\leqslant n}
\begin{vmatrix}
\delta_{i_1}\dot{a}_{i_1}(t) & \delta_{i_1}\ddot{a}_{i_1}(t) & \cdots & \delta_{i_1}a_{i_1}^{(k)}(t) \\[0.7em]
\delta_{i_2}\dot{a}_{i_2}(t) & \delta_{i_2}\ddot{a}_{i_2}(t) & \cdots & \delta_{i_2}a_{i_2}^{(k)}(t) \\[0.7em]
\vdots & \vdots & \ddots & \vdots \\[0.7em]
\delta_{i_k}\dot{a}_{i_k}(t) & \delta_{i_k}\ddot{a}_{i_k}(t) & \cdots & \delta_{i_k}a_{i_k}^{(k)}(t)
\end{vmatrix}^2
\displaybreak[0]
\\&=\sum_{1\leqslant i_1<i_2<\cdots <i_k\leqslant n}
\left\{ \left( \prod_{p=1}^k \delta_{i_p} \right)
\begin{vmatrix}
\dot{a}_{i_1}(t) & \ddot{a}_{i_1}(t) & \cdots & a_{i_1}^{(k)}(t) \\[0.7em]
\dot{a}_{i_2}(t) & \ddot{a}_{i_2}(t) & \cdots & a_{i_2}^{(k)}(t) \\[0.7em]
\vdots & \vdots & \ddots & \vdots \\[0.7em]
\dot{a}_{i_k}(t) & \ddot{a}_{i_k}(t) & \cdots & a_{i_k}^{(k)}(t)
\end{vmatrix} \right\}^2
\displaybreak[0]
\\&\leqslant \delta_1^{2k} \sum_{1\leqslant i_1<i_2<\cdots <i_k\leqslant n}
\begin{vmatrix}
\dot{a}_{i_1}(t) & \ddot{a}_{i_1}(t) & \cdots & a_{i_1}^{(k)}(t) \\[0.7em]
\dot{a}_{i_2}(t) & \ddot{a}_{i_2}(t) & \cdots & a_{i_2}^{(k)}(t) \\[0.7em]
\vdots & \vdots & \ddots & \vdots \\[0.7em]
\dot{a}_{i_k}(t) & \ddot{a}_{i_k}(t) & \cdots & a_{i_k}^{(k)}(t)
\end{vmatrix}^2
=\delta_1^{2k} V_{a,k}^2(t).
\end{align*}
Similarly, we have $V_{b,k}^2(t) \geqslant \delta_n^{2k} V_{a,k}^2(t)$.
Hence
\begin{align}\label{ineq volume}
\delta_n^{k}\leqslant \frac{V_{b,k}(t)}{V_{a,k}(t)} \leqslant \delta_1^{k} \quad (k=1,2,\cdots,m).
\end{align}
Due to the convention that $V_{a,0}(t)=V_{b,0}(t)=1$, inequality (\ref{ineq volume}) also holds for $k=0$.
\par
Finally, using Proposition \ref{prop curvature} and equality (\ref{align eq. curvature}), we have
\begin{align*}
\frac{\kappa_{v,i}(t)}{\kappa_{r,i}(t)}
&=\frac{\kappa_{b,i}(t)}{\kappa_{a,i}(t)}
=\frac{V_{b,i+1}(t)}{V_{a,i+1}(t)} \frac{V_{b,i-1}(t)}{V_{a,i-1}(t)} \frac{V_{a,1}(t)}{V_{b,1}(t)} \left(\frac{V_{a,i}(t)}{V_{b,i}(t)}\right)^2
\\ &\in\left[ \frac{ \delta_n^{i+1} \delta_n^{i-1} }{ \delta_1 \delta_1^{2i} }, ~
\frac{ \delta_1^{i+1} \delta_1^{i-1} }{ \delta_n \delta_n^{2i} } \right]
=\left[ \frac{ \delta_n^{2i} }{ \delta_1^{2i+1} }, ~\frac{ \delta_1^{2i} }{ \delta_n^{2i+1} } \right],
\end{align*}
namely,
\begin{align*}
\frac{ \delta_n^{2i} }{ \delta_1^{2i+1} }\,\kappa_{r,i}(t)
\leqslant \kappa_{v,i}(t)
\leqslant \frac{ \delta_1^{2i} }{ \delta_n^{2i+1} }\,\kappa_{r,i}(t)
\quad (i=1,2,\cdots,m-1).
\end{align*}
It follows that
\begin{align*}
\lim\limits_{t\to+\infty}\kappa_{v,i}(t)=0 &\iff \lim\limits_{t\to+\infty}\kappa_{r,i}(t)=0, \\
\lim\limits_{t\to+\infty}\kappa_{v,i}(t)=+\infty &\iff \lim\limits_{t\to+\infty}\kappa_{r,i}(t)=+\infty, \\
\kappa_{v,i}(t) ~\text{is a bounded function} &\iff \kappa_{r,i}(t) ~\text{is a bounded function}.
\end{align*}
This completes the proof of Theorem \ref{thm relation}.
\end{proof}

\section{Real Jordan Blocks}\label{Section Jordan Blocks}

In order to prove Theorem \ref{thm main}, we can transform $A$ to its real Jordan canonical form by using Proposition \ref{prop Jordan1} and Theorem \ref{thm relation}, and focus on the case of real Jordan canonical form.
\par
Assume that $A$ is a matrix in real Jordan canonical form,
then $A$ is a block diagonal matrix with the following four types of real Jordan blocks:

(1) $1\times 1$ block with real eigenvalue
\begin{align*}
\begin{pmatrix}\lambda\end{pmatrix}_{1\times 1} \quad (\lambda\in\mathbb{R}),
\end{align*}

(2) $p\times p ~(p>1)$ block with real eigenvalue
\begin{align}\label{align RH}
\begin{pmatrix}
\lambda &1&&&\\[0.3em]
&\lambda &1&&\\[-0.3em]
&&\lambda &\ddots&\\[-0.3em]
&&&\ddots &1\\[0.2em]
&&&&\lambda
\end{pmatrix}_{p\times p}
\quad (\lambda\in\mathbb{R}),
\end{align}

(3) $2\times 2$ block with complex eigenvalues
\begin{align}\label{align C2}
\begin{pmatrix}a&b\\-b&a\end{pmatrix}
\quad (a, b\in\mathbb{R},\ b>0),
\end{align}

(4) $2m\times 2m ~(m>1)$ block with complex eigenvalues
\begin{align}\label{align CH}
\begin{pmatrix}
\Lambda & I_2&&&\\[0.3em]
&\Lambda & I_2&&\\[-0.3em]
&&\Lambda &\ddots&\\[-0.3em]
&&&\ddots & I_2\\[0.2em]
&&&&\Lambda
\end{pmatrix}_{2m\times 2m},
\end{align}
where
$\Lambda=
\begin{pmatrix}
 a&b\\
-b&a
\end{pmatrix},
I_2=
\begin{pmatrix}
1&0\\
0&1
\end{pmatrix}
~ (a,b\in\mathbb{R},\text{and}\ b>0).$

\begin{rem}
In the remainder of this paper, we call these four types of real Jordan blocks R1, RH, C2 and CH block for short, respectively.
\end{rem}
We examine these four types of blocks in the following subsections.

\subsection{R1 Block and Real Diagonal Matrix}\label{subsection R1}

\

The case of R1 block is trivial. To study the general case, assume that the matrix $A$ is a real diagonal matrix, namely,
\begin{align*}
A=\mathrm{diag}\{ \lambda_1, \lambda_2, \cdots, \lambda_n \}
\quad (\lambda_1, \lambda_2, \cdots, \lambda_n\in\mathbb{R}).
\end{align*}

Write $r(0)=\left(r_{10},r_{20},\cdots,r_{n0}\right)^\mathrm{T}$,
where $r_{i0}\neq0$ for $i=1,2,\cdots,n$.
Noting that
\begin{align*}
A^2=\mathrm{diag}\{\lambda_1^2, \lambda_2^2, \cdots, \lambda_n^2\}, \quad
\mathrm{e}^{tA}=
\mathrm{diag}\{\mathrm{e}^{\lambda_1 t}, \mathrm{e}^{\lambda_2 t}, \cdots, \mathrm{e}^{\lambda_n t}\},
\end{align*}
we have
\begin{align*}
& r(t)=\mathrm{e}^{tA}r(0)
=\left( \mathrm{e}^{\lambda_1 t}r_{10}, \mathrm{e}^{\lambda_2 t}r_{20}, \cdots, \mathrm{e}^{\lambda_n t}r_{n0} \right)^\mathrm{T},
\\& \dot{r}(t)=Ar(t)
=\left( \lambda_1 \mathrm{e}^{\lambda_1 t}r_{10}, \lambda_2 \mathrm{e}^{\lambda_2 t}r_{20}, \cdots, \lambda_n \mathrm{e}^{\lambda_n t}r_{n0} \right)^\mathrm{T},
\\& \ddot{r}(t)=A^2r(t)
=\left( \lambda_1^2 \mathrm{e}^{\lambda_1 t}r_{10}, \lambda_2^2 \mathrm{e}^{\lambda_2 t}r_{20}, \cdots, \lambda_n^2 \mathrm{e}^{\lambda_n t}r_{n0} \right)^\mathrm{T},
\end{align*}
and
\begin{align*}
\dot{r}_i(t)=\lambda_i \mathrm{e}^{\lambda_i t}r_{i0}, \quad
\ddot{r}_i(t)=\lambda_i^2 \mathrm{e}^{\lambda_i t}r_{i0} \quad
(i=1,2,\cdots,n).
\end{align*}

By formula (\ref{align first curvature}), the square of the first curvature $\kappa(t)$ of $r(t)$ is
\begin{align}\label{align R1 k^2}
\kappa^2(t)
=\frac{V_2^2(t)}{V_1^6(t)}
=\frac{\sum_{1\leqslant i<j\leqslant n}
\left\{\lambda_i \lambda_j(\lambda_j-\lambda_i)r_{i0}r_{j0}\right\}^2 \mathrm{e}^{2(\lambda_i+\lambda_j)t}}
{\left\{\sum_{k=1}^{n}
(\lambda_k r_{k0})^2 \mathrm{e}^{2\lambda_k t}\right\}^3}\quad (n\geqslant 2).
\end{align}

Set
\begin{align*}
\lambda_\mathrm{I}=\max\{ \sigma(A) \backslash \left\{ 0 \right\} \},
\quad
 \lambda_\mathrm{II}=\max\{ \sigma(A) \backslash \left\{ 0, \lambda_\mathrm{I} \right\} \}.
\end{align*}

\begin{rem}
(1) If $A=0_{n\times n}$, then every trajectory $r(t)=\mathrm{e}^{tA}r(0)=r(0)$ is a constant point, and
we have $\kappa(t)\equiv 0$.
\par
(2) By (\ref{align R1 k^2}), if the eigenvalues of $A$ are only $\lambda_\mathrm{I}$ and 0,
then $\kappa(t)\equiv 0$.
\par
The above two cases do not affect the proof of Theorem \ref{thm main}.
\end{rem}

We observe the limit of $\kappa(t)$ as $t\to+\infty$ by comparing the exponents of $\mathrm{e}$ in the numerator $V_2^2(t)$ and denominator $V_1^6(t)$ of $\kappa^2(t)$.
Let $\eta$ and $\theta$ denote the maximum values of $\mu$ in the terms of the form $\mathrm{e}^{\mu t}$ in $V_2^2(t)$ and $V_1^6(t)$, respectively.
Then by (\ref{align R1 k^2}), we have
\begin{align*}
\eta=2\left(\lambda_\mathrm{I}+\lambda_\mathrm{II}\right), \quad \theta=6\lambda_\mathrm{I}.
\end{align*}
It follows that
\begin{align*}
\lim\limits_{t\to+\infty}\kappa(t)=0 \iff \eta<\theta \iff 2\lambda_\mathrm{I}>\lambda_\mathrm{II}, \\
\lim\limits_{t\to+\infty}\kappa(t)=C \iff \eta=\theta \iff 2\lambda_\mathrm{I}=\lambda_\mathrm{II}, \\
\lim\limits_{t\to+\infty}\kappa(t)=+\infty \iff \eta>\theta \iff 2\lambda_\mathrm{I}<\lambda_\mathrm{II},
\end{align*}
where $C>0$ is a constant depending on the initial value $r(0)=r_0 ~(r_{i0}\neq0$ for $i=1,2,\cdots,n)$.

\begin{rem}
We know that the eigenvalues of diagonal matrix $A$ correspond only to the $1\times1$ Jordan blocks.
By Proposition \ref{prop asy.stable}, if $A$ is a diagonal matrix, then the zero solution of the system (\ref{system1}) is stable if and only if
$
\mathrm{Re}\{\lambda_i(A)\}\leqslant 0 ~ (i=1,2,\cdots,n).
$
\end{rem}

\begin{rem}
We notice that the initial value $r(0)=\left(r_{10},r_{20},\cdots,r_{n0}\right)^\mathrm{T}$ may affect the first curvature $\kappa(t)$ of curve $r(t)$.
For simplicity, in the calculations below,
we always assume that $r_{i0}\neq0$ for $i=1,2,\cdots,n$.
It will be seen later (see subsection \ref{subsection results} of Section \ref{Section General}) that this assumption does not affect the proof of Theorem \ref{thm main}.
\end{rem}

\begin{rem}
By (\ref{align R1 k^2}), for any given real diagonal matrix $A$,
if for some initial value $r(0)\in\mathbb{R}^n$ that satisfies $\prod_{i=1}^n {r_{i0}}\neq0$,
such that $\lim\limits_{t\to+\infty}\kappa(t)=0$ (or $+\infty$, or a constant $C>0$, respectively),
then for an arbitrary $r(0)\in\mathbb{R}^n$ satisfying $\prod_{i=1}^n {r_{i0}}\neq0$,
we still have $\lim\limits_{t\to+\infty}\kappa(t)=0$ (or $+\infty$, or a constant $\tilde{C}>0$, respectively).

There are similar results for the following cases in Section \ref{Section Jordan Blocks} and \ref{Section General}.
In fact, we have Theorem \ref{thm initial} and Corollary \ref{cor initial} in Section \ref{Section General}.
\end{rem}

Therefore, we have the following results.

(1) The zero solution of the system is unstable
\begin{align*}
\iff \lambda_\mathrm{I}>0
\Longrightarrow 2\lambda_\mathrm{I}>\lambda_\mathrm{II}
\iff \lim\limits_{t\to+\infty}\kappa(t)=0.
\end{align*}
Hence, if
$\lim\limits_{t\to+\infty}\kappa(t)\neq0 ~ \text{or} ~ \lim\limits_{t\to+\infty}\kappa(t) ~ \text{does not exist}$,
then the zero solution of the system is stable.

(2) The zero solution of the system is not asymptotically stable
\begin{align*}
&\iff \exists \lambda_i=0 ~ \text{or} ~ \exists \lambda_i>0 ~ (i\in\{1,2,\cdots,n\})
\\&\iff \det A=0 ~ \text{or} ~ \lambda_\mathrm{I}>0
\\& ~ \Longrightarrow ~ \det A=0 ~ \text{or} ~ \lim\limits_{t\to+\infty}\kappa(t)=0.
\end{align*}
Hence, if $\det A\neq 0$, and $\lim\limits_{t\to+\infty}\kappa(t)\neq0$ (or $\lim\limits_{t\to+\infty}\kappa(t)$ does not exist),
then the zero solution of the system is asymptotically stable.

Thus, for the case of $A$ is a real diagonal matrix, we obtain the following result.
\begin{prop}
Under the assumptions of Theorem \ref{thm main}, together with the assumption that $A$ is a real diagonal matrix,
for any given initial value $r(0)\in \mathbb{R}^n, ~\mathrm{s.t.,}~ \prod_{i=1}^n {r_{i0}}\neq0$, we have
\par
$(1)$ if $\lim\limits_{t\to+\infty}\kappa(t)\neq0$ or $\lim\limits_{t\to+\infty}\kappa(t)$ does not exist, then the zero solution of the system is stable;
\par
$(2)$ if $A$ is invertible, and $\lim\limits_{t\to+\infty}\kappa(t)\neq0$ or $\lim\limits_{t\to+\infty}\kappa(t)$ does not exist, then the zero solution of the system is asymptotically stable.
\end{prop}

\subsection{RH Block}\label{subsection RH}

\

Let $A$ be $p\times p$ matrix (\ref{align RH}).
Then
\begin{align}\label{align RH exp}
A^2=
\begin{pmatrix}
\lambda^2 & 2\lambda & 1 && \\
& \lambda^2 & 2\lambda & \ddots & \\
&& \lambda^2 & \ddots & 1 \\
&&& \ddots & 2\lambda \\[0.5em]
&&&& \lambda^2
\end{pmatrix}_{p\times p},
\quad
\mathrm{e}^{tA}=
\mathrm{e}^{\lambda t}
\begin{pmatrix}
1 & t & \frac{t^2}{2!} & \frac{t^3}{3!} & \cdots & \frac{t^{p-1}}{(p-1)!} \\[0.5em]
  & 1 & t & \frac{t^2}{2!} & \cdots & \frac{t^{p-2}}{(p-2)!} \\[0.5em]
  &   & 1 & t & \cdots & \frac{t^{p-3}}{(p-3)!} \\[0.5em]
  &   &   & \ddots & \ddots & \vdots \\[0.5em]
&&&& 1 & t \\[0.5em]
&&&&& 1
\end{pmatrix}.
\end{align}

By (\ref{align solution}) and the expression of $\mathrm{e}^{tA}$ in (\ref{align RH exp}), we have
\begin{align*}
r_k(t)=\mathrm{e}^{\lambda t} \sum_{l=0}^{p-k} \frac{r_{k+l,0}}{l!}\, t^l=\mathrm{e}^{\lambda t} P_k(t) \quad (k=1,2,\cdots,p),
\end{align*}
where $r_k(t)$ denotes the $k$th coordinate of $r(t)$, $r_{k0}$ denotes the $k$th coordinate of $r(0)$, and
\begin{align*}
P_k(t)=\sum_{l=0}^{p-k} \frac{r_{k+l,0}}{l!}\, t^l
\end{align*}
is a polynomial in $t$, and $\deg(P_k(t))=p-k$.
\par
For convenience, we have conventions that $r_{p+1}(t)=r_{p+2}(t)=0$, and $P_{p+1}(t)=P_{p+2}(t)=0$.

Since $\dot{r}(t)=Ar(t)$, we have
\begin{align}\label{align RH dot r}
 \dot r_k(t)
 =\lambda r_k(t) + r_{k+1}(t)
 =\mathrm{e}^{\lambda t} \left( \lambda P_k(t) + P_{k+1}(t) \right)\quad(k=1,2,\cdots,p).
\end{align}
Hence
\begin{align*}
 \sum_{k=1}^{p}\dot r_k^2(t)
 =\mathrm{e}^{2\lambda t} h(t),
\end{align*}
where $h(t)=\sum_{k=1}^{p} \left( \lambda P_k(t) + P_{k+1}(t) \right)^2$ is a polynomial in $t$,
and $\deg(h(t))=\left\{
\begin{aligned}
&2(p-1), &\lambda\neq0,\\
&2(p-2), &\lambda=0.
\end{aligned}\right.$

Therefore, the denominator of $\kappa^2(t)$ is
\begin{align}\label{align RH V_1^6}
V_1^6(t)
 =\left( \sum_{k=1}^{p}\dot r_k^2(t) \right)^3
 =\mathrm{e}^{6\lambda t} g(t),
\end{align}
where $g(t)=h^3(t)$, and
\begin{align}\label{align RH deg g}
\deg(g(t))=3\deg(h(t))=\left\{
\begin{aligned}
&6(p-1), &\lambda\neq0,\\
&6(p-2), &\lambda=0.
\end{aligned}\right.
\end{align}
\par
From $\ddot{r}(t)=A^2r(t)$ and (\ref{align RH exp}), we have
\begin{align}\label{align RH ddot r}
  \ddot r_k(t)
 =\lambda^2 r_k(t) + 2\lambda r_{k+1}(t) + r_{k+2}(t)
 =\mathrm{e}^{\lambda t} \left( \lambda^2 P_k(t) + 2\lambda P_{k+1}(t) + P_{k+2}(t) \right)
\end{align}
for $k=1,2,\cdots,p$.

By substituting (\ref{align RH dot r}) and (\ref{align RH ddot r}) into (\ref{align volume 2}),
we obtain the numerator of $\kappa^2(t)$
\begin{align}\label{align RH V_2^2}
V_2^2(t)
=
\sum_{1\leqslant i<j\leqslant p}
\begin{vmatrix}
\dot{r}_{i}(t) & \ddot{r}_{i}(t) \\[0.7em]
\dot{r}_{j}(t) & \ddot{r}_{j}(t)
\end{vmatrix}^2
=
\mathrm{e}^{4\lambda t} f(t),
\end{align}
where
$
f(t)=\sum_{1\leqslant i<j\leqslant p}
\left( \lambda^2
\begin{vmatrix}
P_{i}(t) & P_{i+1}(t) \\[0.7em]
P_{j}(t) & P_{j+1}(t)
\end{vmatrix}
+\lambda
\begin{vmatrix}
P_{i}(t) & P_{i+2}(t) \\[0.7em]
P_{j}(t) & P_{j+2}(t)
\end{vmatrix}
+
\begin{vmatrix}
P_{i+1}(t) & P_{i+2}(t) \\[0.7em]
P_{j+1}(t) & P_{j+2}(t)
\end{vmatrix}
\right)^2
$
is a polynomial in $t$, whose degree is shown in the following remark.

\begin{rem}\label{rem RH deg f}
(1) For $p=2$,
\begin{align*}
V_2^2(t)=
\begin{vmatrix}
\dot{r}_{1}(t) & \ddot{r}_{2}(t) \\[0.7em]
\dot{r}_{1}(t) & \ddot{r}_{2}(t)
\end{vmatrix}^2
=\lambda^4 r_2^4(t)
=\mathrm{e}^{4\lambda t} \lambda^4 r_{20}^4.
\end{align*}

(2) For $p\geqslant 3$,
if $\lambda\neq 0$, then
\begin{align*}
\deg(f(t))
=
\deg\left(\left(
\lambda^2
\begin{vmatrix}
P_{1}(t) & P_{2}(t) \\[0.7em]
P_{2}(t) & P_{3}(t)
\end{vmatrix}
\right)^2 \right)
=
2\deg\left(
\begin{vmatrix}
P_{1}(t) & P_{2}(t) \\[0.7em]
P_{2}(t) & P_{3}(t)
\end{vmatrix}
\right)
=4(p-2);
\end{align*}
if $\lambda=0$, then
\begin{align*}
\deg(f(t))
=
\deg\left(
\begin{vmatrix}
P_{2}(t) & P_{3}(t) \\[0.7em]
P_{3}(t) & P_{4}(t)
\end{vmatrix}^2
\right)
=
2\deg\left(
\begin{vmatrix}
P_{2}(t) & P_{3}(t) \\[0.7em]
P_{3}(t) & P_{4}(t)
\end{vmatrix}
\right)
=4(p-3).
\end{align*}
(If $p=3$ and $\lambda=0$, then $r_4(t)=P_4(t)=0$, but this does not affect the above results.)
\end{rem}
By (\ref{align RH V_1^6}) and (\ref{align RH V_2^2}), the square of the first curvature is
\begin{align*}
\kappa^2(t)
=\frac{V_2^2(t)}{V_1^6(t)}
=\frac{ \mathrm{e}^{4\lambda t} f(t) }{ \mathrm{e}^{6\lambda t} g(t) }
=\frac{ f(t) }{ \mathrm{e}^{2\lambda t} g(t) },
\end{align*}
where the degrees of polynomials $f(t)$ and $g(t)$ are given in Remark \ref{rem RH deg f} and (\ref{align RH deg g}), respectively.

(1) If $\lambda\neq 0$, then
\begin{align*}
\lim\limits_{t\to+\infty}\kappa^2(t)
=\lim\limits_{t\to+\infty}\frac{ f(t) }{ \mathrm{e}^{2\lambda t} g(t) }
=\left\{
\begin{aligned}
&0, &\lambda>0,\\
&+\infty, &\lambda<0.
\end{aligned}\right.
\end{align*}

(2) If $\lambda=0$ and $p=2$, then $\kappa(t)\equiv 0$;
if $\lambda=0$ and $p\geqslant 3$, then
\begin{align*}
\deg(f(t))=4(p-3)<6(p-2)=\deg(g(t)),
\end{align*}
thus
\begin{align*}
\lim\limits_{t\to+\infty}\kappa^2(t)
=\lim\limits_{t\to+\infty}\frac{ f(t) }{ g(t) }
=0.
\end{align*}

In summary,
\begin{align*}
\lim\limits_{t\to+\infty}\kappa(t)
=\left\{
\begin{aligned}
&0, &\lambda\geqslant0,\\
&+\infty, &\lambda<0.
\end{aligned}\right.
\end{align*}
\par
By Proposition \ref{prop asy.stable}, we obtain the following proposition.
\begin{prop}
Under the assumptions of Theorem \ref{thm main}, together with the assumption that $A$ is a $p\times p$ RH block with eigenvalue $\lambda$,
for any given initial value $r(0)\in \mathbb{R}^n, ~\mathrm{s.t.,}~ \prod_{i=1}^p {r_{i0}}\neq0$, we have
\begin{align*}
\lim\limits_{t\to+\infty}\kappa(t)=+\infty
\iff \lambda<0
&\iff \text{the zero solution of the system is stable}
\\&\iff \text{the zero solution of the system is asymptotically stable};
\end{align*}
conversely,
\begin{align*}
\lim\limits_{t\to+\infty}\kappa(t)=0
\iff \lambda\geqslant 0
&\iff \text{the zero solution of the system is unstable}
\\&\iff \text{the zero solution of the system is not asymptotically stable}.
\end{align*}
\end{prop}

\subsection{C2 Block}\label{subsection C2}

\

Let $A$ be matrix (\ref{align C2}).
Then
\begin{align}\label{align C2 exp}
A^2=
\begin{pmatrix}
a^2-b^2 & 2ab \\[0.7em]
-2ab & a^2-b^2
\end{pmatrix},
\quad
\mathrm{e}^{tA}=
\mathrm{e}^{at}
\begin{pmatrix}
\cos bt & \sin bt \\[0.7em]
-\sin bt & \cos bt
\end{pmatrix}.
\end{align}
By (\ref{align solution}) and the expression of $\mathrm{e}^{tA}$ in (\ref{align C2 exp}), we have
\begin{align*}
r_1(t)=\mathrm{e}^{at}T_1(t),\quad
r_2(t)=\mathrm{e}^{at}T_2(t),
\end{align*}
where $T_1(t)=r_{10}\cos{bt}+r_{20}\sin{bt}$, and $T_2(t)=-r_{10}\sin{bt}+r_{20}\cos{bt}$.
\par
Since $\dot{r}(t)=Ar(t)$, we have
\begin{align}\label{align C2 dot r}
& \dot r_1(t)=ar_1(t)+br_2(t)=\mathrm{e}^{at}(aT_1(t)+bT_2(t)), \\
& \dot r_2(t)=-br_1(t)+ar_2(t)=\mathrm{e}^{at}(-bT_1(t)+aT_2(t)). \nonumber
\end{align}
Hence
\begin{align*}
& \dot r_1^2(t)+\dot r_2^2(t)
=\mathrm{e}^{2at} \left(a^2+b^2\right) \left(r_{10}^2+r_{20}^2\right).
\end{align*}

From $\ddot{r}(t)=A^2r(t)$ and (\ref{align C2 exp}), we have
\begin{align}\label{align C2 ddot r}
& \ddot r_1(t)=\left(a^2-b^2\right)r_1(t)+2abr_2(t)=\mathrm{e}^{at}\left[\left(a^2-b^2\right)T_1(t)+2abT_2(t)\right], \\
& \ddot r_2(t)=-2abr_1(t)+\left(a^2-b^2\right)r_2(t)=\mathrm{e}^{at}\left[-2abT_1(t)+\left(a^2-b^2\right)T_2(t)\right]. \nonumber
\end{align}

For C2 block, Wang et al. \cite{Wang} has given the relationship between the curvature and stability.

\begin{prop}[\!\!\cite{Wang}]
Under the assumptions of Theorem \ref{thm main}, additionally assuming that $A$ is the C2 block (\ref{align C2}),
for any given $r(0)\in \mathbb{R}^2\backslash \{0\}$, we have
\begin{align*}
\lim\limits_{t\to+\infty}\kappa(t)=0
\iff a>0
&\iff \text{the zero solution is unstable},
\\
\lim\limits_{t\to+\infty}\kappa(t)=C
\iff a=0
&\iff \text{the zero solution is stable, but not asymptotically stable},
\\
\lim\limits_{t\to+\infty}\kappa(t)=+\infty
\iff a<0
&\iff \text{the zero solution is asymptotically stable},
\end{align*}
where $C>0$ is a constant depending on the initial value $r(0)$.
\end{prop}

\subsection{CH Block}\label{subsection CH}

\

Let $A$ be $2m\times 2m$ matrix (\ref{align CH}). Then
\begin{align}\label{align CH exp}
A^2=
\begin{pmatrix}
\Lambda^2 & 2\Lambda & I_2 && \\
& \Lambda^2 & 2\Lambda & \ddots & \\
&& \Lambda^2 & \ddots & I_2 \\[-0.1em]
&&& \ddots & 2\Lambda \\[0.5em]
&&&& \Lambda^2
\end{pmatrix}_{2m\times 2m},
\quad
\mathrm{e}^{tA}=
\mathrm{e}^{a t}
\begin{pmatrix}
R & tR & \frac{t^2}{2!}R & \frac{t^3}{3!}R & \cdots & \frac{t^{m-1}}{(m-1)!}R \\[0.7em]
  & R & tR & \frac{t^2}{2!}R & \cdots & \frac{t^{m-2}}{(m-2)!}R \\[0.7em]
  &   & R & tR & \cdots & \frac{t^{m-3}}{(m-3)!}R \\[0.7em]
  &   &   & \ddots & \ddots & \vdots \\[0.7em]
&&&& R & tR \\[0.7em]
&&&&& R
\end{pmatrix},
\end{align}
where
$
\Lambda^2=
\begin{pmatrix}
a^2-b^2 & 2ab \\[0.7em]
-2ab & a^2-b^2
\end{pmatrix},
$
and the rotation matrix
$
R=
\begin{pmatrix}
\cos bt & \sin bt \\[0.7em]
-\sin bt & \cos bt
\end{pmatrix}.
$

\begin{rem}
$
\det A=\left(a^2+b^2\right)^m>0.
$
\end{rem}

\begin{rem}
The matrix $A$ is similar to
$
\begin{pmatrix}
J_m(a+b\sqrt{-1}) & 0 \\[0.7em]
0 & J_m(a-b\sqrt{-1})
\end{pmatrix},
$
where $J_m(\lambda)$ denotes the $m\times m$ Jordan block with eigenvalue $\lambda\in\mathbb{C}$.
By Proposition \ref{prop asy.stable}, if $a=0$, then the zero solution of system (\ref{system1}) is unstable.
Thus, the zero solution of the system is stable if and only if $a<0$.
\end{rem}

Write
\begin{align*}
r(t)
&=\left(r_{1}(t),r_{2}(t),r_{3}(t),\cdots,r_{2m-1}(t),r_{2m}(t)\right)^\mathrm{T} \\
&=\left(r_{11}(t),r_{12}(t),r_{21}(t),r_{22}(t),\cdots,r_{m1}(t),r_{m2}(t)\right)^\mathrm{T},
\end{align*}
and
$
r(0)=\left(r_{11,0},~ r_{12,0},~ r_{21,0},~ r_{22,0},~ \cdots,~ r_{m1,0},~ r_{m2,0}\right)^\mathrm{T}
$
which satisfies
$
\prod_{i=1}^m\prod_{j=1}^2 r_{ij,0}\neq0.
$

By (\ref{align solution}) and the expression of $\mathrm{e}^{tA}$ in (\ref{align CH exp}), we have
\begin{align*}
r_{ij}(t)=\mathrm{e}^{at}T_{ij}(t)
\quad (i=1,2,\cdots,m;~j=1,2),
\end{align*}
where
\begin{align*}
T_{i1}(t)&=\sum_{k=0}^{m-i}\frac{t^k}{k!}\left(r_{2i+2k-1,0}\cos{bt}+r_{2i+2k,0}\sin{bt}\right), \\
T_{i2}(t)&=\sum_{k=0}^{m-i}\frac{t^k}{k!}\left(-r_{2i+2k-1,0}\sin{bt}+r_{2i+2k,0}\cos{bt}\right),
\end{align*}
and we have a convention that if $i>m$, then
$r_{ij}(t)=0 ~(j=1,2)$.
Hence
\begin{align}\label{align CH r_1^2+r_2^2}
r_{i1}^2(t)+r_{i2}^2(t)
=\mathrm{e}^{2at} \left(T_{i1}^2(t)+T_{i2}^2(t)\right)
=\mathrm{e}^{2at} \left(\hat{C} t^{2(m-i)} + \sum_{\varphi=0}^{2(m-i)-1}\hat{B}_\varphi(t) t^{\varphi} \right),
\end{align}
where
$
\hat{C}=\frac{r_{m1,0}^2+r_{m2,0}^2}{\left[(m-1)!\right]^2}>0
$
is a constant.

By (\ref{system1}), we have
\begin{align}\label{align CH dot r}
& \dot r_{i1}(t)=ar_{i1}(t)+br_{i2}(t)+r_{i+1,1}(t)=\mathrm{e}^{at}\left(aT_{i1}(t)+bT_{i2}(t)+T_{i+1,1}(t)\right), \\
& \dot r_{i2}(t)=-br_{i1}(t)+ar_{i2}(t)+r_{i+1,2}(t)=\mathrm{e}^{at}\left(-bT_{i1}(t)+aT_{i2}(t)+T_{i+1,2}(t)\right) \nonumber
\end{align}
for $i=1,2,\cdots,m$. Hence
\begin{align*}
V_1^2(t)
=&~
\sum_{k=1}^{2m} \dot r_k^2(t)
=
\sum_{i=1}^{m} \left( \dot r_{i1}^2(t) + \dot r_{i2}^2(t) \right)
\displaybreak[0]
\\=&~
\mathrm{e}^{2at} \sum_{i=1}^{m} \left\{
\left( aT_{i1}(t)+bT_{i2}(t)+T_{i+1,1}(t) \right)^2
+\left( -bT_{i1}(t)+aT_{i2}(t)+T_{i+1,2}(t) \right)^2
\right\}
\displaybreak[0]
\\=&~
\mathrm{e}^{2at}\left\{
(a^2+b^2)\left( T_{11}^2(t)+T_{12}^2(t) \right)
+\sum_{i=2}^{m} (a^2+b^2+1)\left( T_{i1}^2(t)+T_{i2}^2(t) \right)
\right. \\ &\left.
+\sum_{i=1}^{m-1} \Big[
2a \left( T_{i1}(t)T_{i+1,1}(t) + T_{i2}(t)T_{i+1,2}(t) \right)
+2b \left( T_{i2}(t)T_{i+1,1}(t) - T_{i1}(t)T_{i+1,2}(t) \right)
\Big]
\right\}.
\end{align*}
Noting that
\begin{align*}
T_{ij}(t)=\sum_{\psi=0}^{m-i} B_{ij,\psi(t)} t^{\psi}
\quad (i=1,2,\cdots,m;~ j=1,2),
\end{align*}
where $B_{ij,\psi(t)}~(i=1,2,\cdots,m;~ j=1,2)$ are bounded trigonometric functions,
we have
\begin{align}\label{align CH V_1^6}
V_1^2(t)=
\mathrm{e}^{2at} \left(Ct^{2m-2} + \sum_{\psi=0}^{2m-3} B_\psi(t) t^{\psi} \right),
\end{align}
where
$
C=\frac{\left(a^2+b^2\right)\left(r_{m1,0}^2+r_{m2,0}^2\right)}{\left[(m-1)!\right]^2}>0
$
is a constant, and
$
B_\psi(t) ~( \psi=0,1,\cdots,2m-3 )
$
are bounded functions.

From $\ddot{r}(t)=A^2r(t)$ and (\ref{align CH exp}), we have
\begin{align}\label{align CH ddot r}
\ddot r_{i1}(t)
&= \left(a^2-b^2\right)r_{i1}(t)+2abr_{i2}(t)+2ar_{i+1,1}(t)+2br_{i+1,2}(t)+r_{i+2,1}(t) \\
&= \mathrm{e}^{at} \left[\left(a^2-b^2\right)T_{i1}(t)+2abT_{i2}(t)+2aT_{i+1,1}(t)+2bT_{i+1,2}(t)+T_{i+2,1}(t) \right], \nonumber\\
\ddot r_{i2}(t)
&= -2abr_{i1}(t)+\left(a^2-b^2\right)r_{i2}(t)-2br_{i+1,1}(t)+2ar_{i+1,2}(t)+r_{i+2,2}(t) \nonumber\\
&= \mathrm{e}^{at} \left[-2abT_{i1}(t)+\left(a^2-b^2\right)T_{i2}(t)-2bT_{i+1,1}(t)+2aT_{i+1,2}(t)+T_{i+2,2}(t) \right] \nonumber
\end{align}
for $i=1,2,\cdots,m$.
Noticing the expressions of (\ref{align CH dot r}) and (\ref{align CH ddot r}), write
\begin{align*}
\dot r_{ij}(t)=\mathrm{e}^{at} T_{ij}^{(\mathrm{I})}(t), \quad
\ddot r_{ij}(t)=\mathrm{e}^{at} T_{ij}^{(\mathrm{II})}(t)
\end{align*}
for $i=1,2,\cdots,m$ and $j=1,2$,
where both $T_{ij}^{(\mathrm{I})}(t)$ and $T_{ij}^{(\mathrm{II})}(t)$
are linear combinations of functions $T_{kl}~(k=1,2,\cdots,m;~ l=1,2)$. Thus
\begin{align}\label{align CH V_2^2}
V_2^2(t)
=
\sum_{1\leqslant k<l\leqslant m} ~
\sum_{p=1}^2 ~
\sum_{q=1}^2
\begin{vmatrix}
\dot{r}_{kp}(t) & \ddot{r}_{kp}(t) \\[0.7em]
\dot{r}_{lq}(t) & \ddot{r}_{lq}(t)
\end{vmatrix}^2
+
\sum_{i=1}^m
\begin{vmatrix}
\dot{r}_{i1}(t) & \ddot{r}_{i1}(t) \\[0.7em]
\dot{r}_{i2}(t) & \ddot{r}_{i2}(t)
\end{vmatrix}^2
=
\mathrm{e}^{4at}f(t),
\end{align}
where
$
f(t)=
\sum_{1\leqslant k<l\leqslant m}
\sum_{p=1}^2
\sum_{q=1}^2
\begin{vmatrix}
T_{kp}^{(\mathrm{I})}(t) & T_{kp}^{(\mathrm{II})}(t) \\[0.7em]
T_{lq}^{(\mathrm{I})}(t) & T_{lq}^{(\mathrm{II})}(t)
\end{vmatrix}^2
+
\sum_{i=1}^m
\begin{vmatrix}
T_{i1}^{(\mathrm{I})}(t) & T_{i1}^{(\mathrm{II})}(t) \\[0.7em]
T_{i2}^{(\mathrm{I})}(t) & T_{i2}^{(\mathrm{II})}(t)
\end{vmatrix}^2.
$

\begin{rem}
The function $f(t)$ can be expressed in the form of
\begin{align}\label{align CH f}
f(t)=
\tilde C t^{4m-4} + \sum_{\varphi=0}^{4m-5} \tilde B_\varphi(t) t^{\varphi},
\end{align}
where
$
\tilde C=\frac{ b^2 \left(a^2+b^2\right)^2 \left(r_{m1,0}^2+r_{m2,0}^2\right)^2 }{\left[(m-1)!\right]^4}>0
$
is a constant, and
$
\tilde B_\varphi(t) ~( \varphi=0,1,\cdots,4m-5 )
$
are bounded functions.
\begin{proof}
By (\ref{align CH r_1^2+r_2^2}),
\begin{align*}
r_{1}^2(t)+r_{2}^2(t)
=\mathrm{e}^{2at} \left\{\frac{r_{m1,0}^2+r_{m2,0}^2}{\left[(m-1)!\right]^2}\, t^{2m-2} + \sum_{\varphi=0}^{2m-3}\hat{B}_\varphi(t) t^{\varphi} \right\}.
\end{align*}
In (\ref{align volume 2}),
we can reach the highest power $t^{4m-4}$ in the expression of $f(t)$ by taking $i=1$ and $j=2$. In fact,
\begin{align}\label{align CH f proof}
&
\begin{vmatrix}
\dot{r}_{1}(t) & \ddot{r}_{1}(t) \\[0.7em]
\dot{r}_{2}(t) & \ddot{r}_{2}(t)
\end{vmatrix}
\displaybreak[0]
\\=&
\begin{vmatrix}
ar_{1}(t)+br_{2}(t)+r_{3}(t) &~ \left(a^2-b^2\right)r_{1}(t)+2abr_{2}(t)+2ar_{3}(t)+2br_{4}(t)+r_{5}(t)  \\[0.7em]
-br_{1}(t)+ar_{2}(t)+r_{4}(t) &~ -2abr_{1}(t)+\left(a^2-b^2\right)r_{2}(t)-2br_{3}(t)+2ar_{4}(t)+r_{6}(t)
\end{vmatrix}
\nonumber\displaybreak[0]\\=&
\begin{vmatrix}
ar_{1}(t)+br_{2}(t) &~ \left(a^2-b^2\right)r_{1}(t)+2abr_{2}(t) \\[0.7em]
-br_{1}(t)+ar_{2}(t) &~ -2abr_{1}(t)+\left(a^2-b^2\right)r_{2}(t)
\end{vmatrix}
+ \mathrm{e}^{2at} \sum_{\varphi=0}^{2m-3} D_\varphi(t) t^{\varphi}
\nonumber\displaybreak[0]\\=&~
\mathrm{e}^{2at}\left\{ -b\left(a^2+b^2\right)\left(r_1^2(t)+r_2^2(t)\right)
+ \sum_{\varphi=0}^{2m-3} D_\varphi(t) t^{\varphi} \right\}
\nonumber\displaybreak[0]\\=&~
\mathrm{e}^{2at}\left\{ \frac{ -b\left(a^2+b^2\right) \left(r_{m1,0}^2+r_{m2,0}^2\right) }{\left[(m-1)!\right]^2} \, t^{2m-2}
+ \sum_{\varphi=0}^{2m-3} \tilde D_\varphi(t) t^{\varphi} \right\},
\nonumber
\end{align}
where both $D_\varphi(t)$ and $\tilde D_\varphi(t)~( \varphi=0,1,\cdots,2m-3)$ are bounded functions.
By substituting (\ref{align CH f proof}) into (\ref{align volume 2}),
we obtain (\ref{align CH f}).
\end{proof}
\end{rem}

By (\ref{align CH V_1^6}), (\ref{align CH V_2^2}) and (\ref{align CH f}), we obtain
\begin{align*}
\kappa^2(t)
=\frac{V_2^2(t)}{V_1^6(t)}
=\frac{ \mathrm{e}^{4at} \left( \tilde C t^{4m-4} + \sum_{\varphi=0}^{4m-5} \tilde B_\varphi(t) t^{\varphi} \right) }
{ \mathrm{e}^{6at} \left(Ct^{2m-2} + \sum_{\psi=0}^{2m-3} B_\psi(t) t^{\psi} \right)^3 }
=\frac{ \tilde C t^{4m-4} + \sum_{\varphi=0}^{4m-5} \tilde B_\varphi(t) t^{\varphi} }
{ \mathrm{e}^{2at} \left(\bar C t^{6m-6} + \sum_{\psi=0}^{6m-7} \bar B_\psi(t) t^{\psi} \right) },
\end{align*}
where $\tilde C>0$ and $\bar C=C^3>0$ are constants,
$\tilde B_\varphi(t) ~( \varphi=0,1,\cdots,4m-5 )$ and
$\bar B_\psi(t) ~( \psi=0,1,\cdots,6m-7 )$
are bounded functions.

Therefore, we have the following results.
\par
(1) For $a>0$, we have $\lim\limits_{t\to+\infty}\kappa(t)=0$.
\par
(2) For $a=0$, we have
\begin{align*}
\lim\limits_{t\to+\infty}\kappa^2(t)
=\frac{ \tilde C t^{4m-4} + \sum_{\varphi=0}^{4m-5} \tilde B_\varphi(t) t^{\varphi} }
{ \bar C t^{6m-6} + \sum_{\psi=0}^{6m-7} \bar B_\psi(t) t^{\psi} }
=0,
\end{align*}
hence $\lim\limits_{t\to+\infty}\kappa(t)=0$.
\par
(3) For $a<0$, we have $\lim\limits_{t\to+\infty}\kappa(t)=+\infty$.
\par
In summary, we obtain the following proposition.
\begin{prop}
Under the assumptions of Theorem \ref{thm main}, additionally assuming that $A$ is a $2m\times 2m$ CH block with eigenvalues $a\pm b\sqrt{-1}$,
for any given $r(0)\in \mathbb{R}^n, ~\mathrm{s.t.,}~ \prod_{i=1}^{2m} {r_{i0}}\neq0$, we have
\begin{align*}
\lim\limits_{t\to+\infty}\kappa(t)=+\infty
\iff a<0
&\iff \text{the zero solution of the system is stable}
\\&\iff \text{the zero solution of the system is asymptotically stable};
\end{align*}
conversely,
\begin{align*}
\lim\limits_{t\to+\infty}\kappa(t)=0
\iff a\geqslant 0
&\iff \text{the zero solution of the system is unstable}
\\&\iff \text{the zero solution of the system is not asymptotically stable}.
\end{align*}
\end{prop}

\section{General Case}\label{Section General}

In this section, we consider the general case, namely, $A$ is an $n\times n$ matrix, and prove Theorem \ref{thm main}. Since $A$ is similar to its real Jordan canonical form, we only need to focus on the case of real Jordan canonical form, and prove the following theorem.

\begin{thm} \label{thm Jordan}
Take the assumptions of Theorem \ref{thm main}, and additionally assume that $A$ is a matrix in real Jordan canonical form.
For any given initial value $r(0)\in\left\{ \left(r_{10},r_{20},\cdots,r_{n0}\right)^\mathrm{T}\in\mathbb{R}^n\Big|\prod_{i=1}^n {r_{i0}}\neq0 \right\}$, we have
\par
$(1)$ if $\lim\limits_{t\to+\infty}\kappa(t)\neq0$ or $\lim\limits_{t\to+\infty}\kappa(t)$ does not exist, then the zero solution of the system is stable;
\par
$(2)$ if $A$ is invertible, and $\lim\limits_{t\to+\infty}\kappa(t)=+\infty$, then the zero solution of the system is asymptotically stable.
\end{thm}

\subsection{Review of Calculation Results}\label{subsection review}

\

Let $A$ be an $n\times n$ matrix in real Jordan canonical form, then $A$ is a block diagonal real matrix whose diagonal consists of R1, RH, C2 and CH blocks.
Through the analysis of Section \ref{Section Jordan Blocks}, we have the following results.\\

(1) For R1 block $\begin{pmatrix}\lambda\end{pmatrix}_{1\times 1} (\lambda\in\mathbb{R})$, we have
\begin{align}\label{align sq.der.1}
r(t)=\mathrm{e}^{\lambda t}r_0, \quad
\dot{r}(t)=\mathrm{e}^{\lambda t} \lambda r_0, \quad
\ddot{r}(t)=\mathrm{e}^{\lambda t} \lambda^2 r_0, \quad
\dot{r}^2(t)=\mathrm{e}^{2\lambda t} \lambda^2 r_0^2.
\end{align}

(2) For $p\times p ~(p>1)$ RH block (\ref{align RH}), we have
\begin{align*}
 r_k(t)=\mathrm{e}^{\lambda t} P_k(t) \quad ( k=1,2,\cdots,p ),
\end{align*}
where $P_k(t)=\sum_{l=0}^{p-k} \frac{r_{k+l,0}}{l!}\, t^l$,
and we have a convention that $r_{p+1}(t)=r_{p+2}(t)=0$. If $1\leqslant k \leqslant p$, then $\deg(P_k(t))=p-k$;
if $k>p$, then $P_k(t)=0$. Hence
\begin{align}\label{align sq.der.2}
 \dot r_k(t)
 =\lambda r_k(t) + r_{k+1}(t)
 =\mathrm{e}^{\lambda t} \left( \lambda P_k(t) + P_{k+1}(t) \right), \nonumber\\
 \sum_{k=1}^{p}\dot r_k^2(t)
 =\mathrm{e}^{2\lambda t} \sum_{k=1}^{p} \left( \lambda P_k(t) + P_{k+1}(t) \right)^2
 =\mathrm{e}^{2\lambda t} h(t),
\end{align}
where $\deg(h(t))=\left\{
\begin{aligned}
&2(p-1), &\lambda\neq0,\\
&2(p-2), &\lambda=0,
\end{aligned}\right.$
~and we have
\begin{align*}
 \ddot r_k(t)
 =\lambda^2 r_k(t) + 2\lambda r_{k+1}(t) + r_{k+2}(t)
 =\mathrm{e}^{\lambda t} \left( \lambda^2 P_k(t) + 2\lambda P_{k+1}(t) + P_{k+2}(t) \right).
\end{align*}

(3) For C2 block (\ref{align C2}), we have
\begin{align*}
r_1(t)=\mathrm{e}^{at}T_1(t),\quad
r_2(t)=\mathrm{e}^{at}T_2(t),
\end{align*}
where $T_1(t)=r_{10}\cos{bt}+r_{20}\sin{bt}$, and $T_2(t)=-r_{10}\sin{bt}+r_{20}\cos{bt}$.
Hence
\begin{align}\label{align sq.der.3}
& \dot r_1(t)=ar_1(t)+br_2(t)=\mathrm{e}^{at}\left(aT_1(t)+bT_2(t)\right), \nonumber\\
& \dot r_2(t)=-br_1(t)+ar_2(t)=\mathrm{e}^{at}\left(-bT_1(t)+aT_2(t)\right), \nonumber\\
& \dot r_1^2(t)+\dot r_2^2(t)=\mathrm{e}^{2at}\left(a^2+b^2\right)\left(r_{10}^2+r_{20}^2\right), \\
& \ddot r_1(t)=\left(a^2-b^2\right)r_1(t)+2abr_2(t)=\mathrm{e}^{at}\left[\left(a^2-b^2\right)T_1(t)+2abT_2(t)\right], \nonumber\\
& \ddot r_2(t)=-2abr_1(t)+\left(a^2-b^2\right)r_2(t)=\mathrm{e}^{at}\left[-2abT_1(t)+\left(a^2-b^2\right)T_2(t)\right]. \nonumber
\end{align}

(4) For $2m\times 2m ~(m>1)$ CH block (\ref{align CH}),
write
\begin{align*}
r(t)=\left(r_{11}(t),r_{12}(t),r_{21}(t),r_{22}(t),\cdots,r_{m1}(t),r_{m2}(t)\right)^\mathrm{T},
\end{align*}
then
\begin{align*}
r_{i1}(t)=\mathrm{e}^{at}T_{i1}(t),\quad
r_{i2}(t)=\mathrm{e}^{at}T_{i2}(t)\quad (i=1,2,\cdots,m),
\end{align*}
where
$T_{i1}(t)=\sum_{k=0}^{m-i}\frac{t^k}{k!}(r_{2i+2k-1,0}\cos{bt}+r_{2i+2k,0}\sin{bt})$,
$T_{i2}(t)=\sum_{k=0}^{m-i}\frac{t^k}{k!}(-r_{2i+2k-1,0}\sin{bt}+r_{2i+2k,0}\cos{bt})$,
and we have a convention that if $i>m$, then
$r_{ij}(t)=0 ~(j=1,2)$.
Hence
\begin{align}
& \dot r_{i1}(t)=ar_{i1}(t)+br_{i2}(t)+r_{i+1,1}(t)=\mathrm{e}^{at}\left(aT_{i1}(t)+bT_{i2}(t)+T_{i+1,1}(t)\right), \nonumber\\
& \dot r_{i2}(t)=-br_{i1}(t)+ar_{i2}(t)+r_{i+1,2}(t)=\mathrm{e}^{at}\left(-bT_{i1}(t)+aT_{i2}(t)+T_{i+1,2}(t)\right)\nonumber
\end{align}
for $i=1,2,\cdots,m$, and
\begin{align}\label{align sq.der.4}
& \sum_{k=1}^{2m} \dot r_k^2(t)=\mathrm{e}^{2at} \left(Ct^{2m-2} + \sum_{\psi=0}^{2m-3} B_\psi(t) t^{\psi} \right),
\end{align}
where
$
C=\frac{\left(a^2+b^2\right)\left(r_{m1,0}^2+r_{m2,0}^2\right)}{\left[(m-1)!\right]^2}>0
$
is a constant, and
$B_\psi(t)~(\psi=0,1,\cdots,2m-3)$
are bounded functions. Moreover,
\begin{align*}
\ddot r_{i1}(t)
&= \left(a^2-b^2\right)r_{i1}(t)+2abr_{i2}(t)+2ar_{i+1,1}(t)+2br_{i+1,2}(t)+r_{i+2,1}(t) \\
&= \mathrm{e}^{at} \left[\left(a^2-b^2\right)T_{i1}(t)+2abT_{i2}(t)+2aT_{i+1,1}(t)+2bT_{i+1,2}(t)+T_{i+2,1}(t) \right], \\
\ddot r_{i2}(t)
&= -2abr_{i1}(t)+\left(a^2-b^2\right)r_{i2}(t)-2br_{i+1,1}(t)+2ar_{i+1,2}(t)+r_{i+2,2}(t) \\
&= \mathrm{e}^{at} \left[-2abT_{i1}(t)+\left(a^2-b^2\right)T_{i2}(t)-2bT_{i+1,1}(t)+2aT_{i+1,2}(t)+T_{i+2,2}(t) \right]
\end{align*}
for $i=1,2,\cdots,m$.

\subsection{Denominator of $\kappa^2(t)$}

\

We examine $V_1^6(t)$, the denominator of $\kappa^2(t)$, in this subsection.

Let $A$ be a matrix in real Jordan canonical form whose diagonal consists of $q$ real Jordan blocks,
where the $i$th block is an $n_i\times n_i$ matrix. Then
\begin{align*}
V_1^2(t)=\sum_{i=1}^q \sum_{k=1}^{n_i} \dot r_{ik}^2(t),
\end{align*}
where $\dot r_{ik}(t)$ denotes the coordinate of $\dot r(t)$ corresponding to the $k$th row of the $i$th real Jordan block.
By (\ref{align sq.der.1}), (\ref{align sq.der.2}), (\ref{align sq.der.3}) and (\ref{align sq.der.4}), we have
\begin{align*}
\sum_{k=1}^{n_i} \dot r_{ik}^2(t)
=\mathrm{e}^{2\mathrm{Re}(\lambda_i)t} g_i(t),
\end{align*}
where $\lambda_i$ is an eigenvalue of the $i$th block,
and the expression of $g_i(t)$ depends on the type of the $i$th real Jordan block.
In fact, we have the following results.

(1) For R1 block,
\begin{align*}
g_i(t)=\lambda_i^2 r_{i1,0}^2=\left\{
\begin{aligned}
&C, &\lambda\neq0,\\
&0, &\lambda=0,
\end{aligned}\right.
\end{align*}
where $C>0$ is a constant.

(2) For $p\times p$ RH block, $g_i(t)$ is a polynomial
\begin{align*}
g_i(t)=\sum_{k=1}^{p} \left( \lambda_i P_k(t) + P_{k+1}(t) \right)^2,
\end{align*}
and $\deg(g_i(t))=\left\{
\begin{aligned}
&2(p-1), &\lambda\neq0,\\
&2(p-2), &\lambda=0.
\end{aligned}\right.$

(3) For C2 block, $g_i(t)$ is the constant
\begin{align*}
g_i(t)=\left(a^2+b^2\right)\left(r_{i1,0}^2+r_{i2,0}^2\right)>0.
\end{align*}

(4) For $2m\times 2m$ CH block,
\begin{align*}
g_i(t)=Ct^{2m-2} + \sum_{\psi=0}^{2m-3} B_\psi(t) t^{\psi},
\end{align*}
where $C>0$ is a constant, and
$B_\psi(t)~(\psi=0,1,\cdots, 2m-3)$
are bounded functions.

Hence the denominator of $\kappa^2(t)$ is
\begin{align*}
V_1^6(t)=\left(\sum_{i=1}^q\mathrm{e}^{2\mathrm{Re}(\lambda_i)t} g_i(t)\right)^3.
\end{align*}
We see that $g_i(t)=0$ if and only if the $i$th block is an R1 block with $\lambda=0$,
namely, the $i$th block is $0_{1\times 1}$,
which causes $\mathrm{e}^{2\mathrm{Re}(\lambda_i)t}$ of the block to vanish in $V_1^6(t)$.
Let $\tilde\sigma(A)$ denote the set of eigenvalues of $A$ which excluding the zero eigenvalues in R1 blocks, and
\begin{align}\label{align M}
M=\max\{\mathrm{Re}(\lambda)| \lambda\in \sigma(A)\}, \quad
\tilde M=\max\{\mathrm{Re}(\lambda)| \lambda\in \tilde\sigma(A)\}.
\end{align}
Then
\begin{align}\label{align V1}
V_1^6(t)
=\mathrm{e}^{6\tilde M t}\left( \bar C t^\xi + \sum_{\psi=0}^{\xi-1} \bar B_\psi{(t)} t^\psi \right)+R(t),
\end{align}
where $\bar C>0$ is a constant, $\bar B_{\psi}(t)~(\psi=0,1,\cdots, \xi-1)$ are bounded functions, and $R(t)$ is a linear combination of terms in the form of $\mathrm{e}^{\mu t} t^\nu B_\omega(t)$,
here $\mu<6\tilde M$, and $B_\omega(t)$ is a bounded function.

Hence we have
\begin{align}\label{align theta}
\theta=6\tilde M,
\end{align}
where $\theta$ denotes the maximum value of $\mu$ in the terms of the form $\mathrm{e}^{\mu t} t^\nu B_\omega(t)$ in $V_1^6(t)$.

\begin{rem}\label{rem sigma}
In (\ref{align V1}), the integer $\xi\geqslant0$, and
\begin{align*}
\xi=\max\{ \xi_R, \xi_C \}.
\end{align*}

(1) If $\tilde M\neq0$, then $\xi_R=6(p-1)$, and $\xi_C=6(m-1)$,
where $p$ denotes the maximum order of R1 or RH blocks with $\tilde M$ as eigenvalue,
and $m$ denotes half of the maximum order of C2 or CH blocks with $\tilde M\pm b\sqrt{-1}~(b\in\mathbb{R})$ as eigenvalues.

(2) If $\tilde M=0$, then the definition of $\xi_C$ is the same as (1),
however, $\xi_R=6(p-2)$, where $p$ denotes the maximum order of RH blocks with eigenvalue $0$.
\end{rem}

\begin{rem}\label{rem M}
Suppose that $A$ is a matrix in real Jordan canonical form.
By (\ref{align M}) and Proposition \ref{prop asy.stable}, we obtain the following results.

(1) The zero solution of system (\ref{system1}) is stable
if and only if $M\leqslant0$,
and the real Jordan blocks whose eigenvalues have zero real parts in the diagonal of $A$ are either R1 or C2.

(2) The zero solution of system (\ref{system1}) is asymptotically stable
if and only if $M<0$.
\end{rem}

\subsection{Numerator of $\kappa^2(t)$}

\

Now, we examine the numerator $V_2^2(t)$ of $\kappa^2(t)$.

By subsection \ref{subsection review}, we see that all coordinates of $\dot r(t)$ and $\ddot r(t)$ of R1, RH, C2 and CH blocks can be expressed in the form of
\begin{align*}
\dot r_{ik}(t)&=\mathrm{e}^{\mathrm{Re}(\lambda_i)t} f_{ik}(t), \\
\ddot r_{ik}(t)&=\mathrm{e}^{\mathrm{Re}(\lambda_i)t} \tilde f_{ik}(t),
\end{align*}
thus
\begin{align}\label{align V2}
\begin{vmatrix}
\dot{r}_{ik}(t) & \ddot{r}_{ik}(t) \\[0.7em]
\dot{r}_{jl}(t) & \ddot{r}_{jl}(t)
\end{vmatrix}^2
=\begin{vmatrix}
\mathrm{e}^{\mathrm{Re}(\lambda_i)t} f_{ik}(t) &~ \mathrm{e}^{\mathrm{Re}(\lambda_i)t} \tilde f_{ik}(t) \\[1em]
\mathrm{e}^{\mathrm{Re}(\lambda_j)t} f_{jl}(t) &~ \mathrm{e}^{\mathrm{Re}(\lambda_j)t} \tilde f_{jl}(t)
\end{vmatrix}^2
=\mathrm{e}^{2\{\mathrm{Re}\left(\lambda_i)+\mathrm{Re}(\lambda_j)\right\}t} F(t),
\end{align}
where
$F(t)=
\begin{vmatrix}
f_{ik}(t) & \tilde f_{ik}(t) \\[0.7em]
f_{jl}(t) & \tilde f_{jl}(t)
\end{vmatrix}^2$
 is a linear combination of terms in the form of $B_\varphi(t)t^\varphi$,
here $B_\varphi(t)$ is a bounded function.
By substituting (\ref{align V2}) into (\ref{align volume 2}), we obtain
\begin{align}\label{align eta}
\eta\leqslant 4M.
\end{align}
where $\eta$ denotes the maximum value of $\mu$ in the terms of the form $\mathrm{e}^{\mu t} t^\nu B_\omega(t)$ in $V_2^2(t)$.

\subsection{Proof of Theorem \ref{thm Jordan}(1)}

\

In this subsection, we prove Theorem \ref{thm Jordan}(1).
\begin{lem}\label{lemma M}
Under the assumptions above, if $M>0$, then $\lim\limits_{t\to+\infty}\kappa(t)=0$.
\end{lem}

\begin{proof}
Suppose $M>0$ in (\ref{align M}), then $\tilde M=M$.
By (\ref{align theta}) and (\ref{align eta}), we have
\begin{align*}
\eta\leqslant 4M<6\tilde M=\theta.
\end{align*}
It follows that
$\lim\limits_{t\to+\infty}\kappa(t)=0$.
\end{proof}

\begin{lem}\label{lemma 1}
Under the assumptions of Theorem \ref{thm Jordan},
if $M=0$,
and there exist RH or CH blocks whose eigenvalues have zero real parts in the diagonal of $A$,
then $\lim\limits_{t\to+\infty}\kappa(t)=0$.
\end{lem}

\begin{proof}
Suppose $M=0$. By the assumption that there exist RH or CH blocks whose eigenvalues have zero real parts in the diagonal of $A$,
we obtain $\tilde\sigma(A)=\sigma(A)$, and $\tilde M=M$. Thus we have
\begin{align*}
\eta\leqslant 4 M=0, \quad
\theta=6\tilde M=0.
\end{align*}
\par
(1) If $\eta<0$, then $\eta<\theta$, hence $\lim\limits_{t\to+\infty}\kappa(t)=0$.
\par
(2) If $\eta=0$, then we compare the highest power of $t$ of terms in the form of $\mathrm{e}^{0t} t^\nu B_\omega(t)$ in the expression of $\kappa^2(t)$, where $B_\omega(t)$ is a bounded function.
In fact, by Remark \ref{rem sigma}, we have
\begin{align}\label{align xi}
\xi=\max\{ 6(p-2), 6(m-1) \}
\end{align}
in $V_1^6(t)$,
where $p$ denotes the maximum order of RH blocks with eigenvalue $0$,
and $m$ denotes half of the maximum order of C2 or CH blocks with $\pm b\sqrt{-1}~(b\in\mathbb{R})$ as eigenvalues.
\par
In $V_2^2(t)$, the highest power of $t$ of terms in the form of $\mathrm{e}^{0t} t^\nu B_\omega(t)$
depends on the orders of RH or CH blocks whose eigenvalues have zero real parts.
\par
For a $p\times p$ RH block with $\lambda=0$, the first coordinates of $\dot r(t)$ and $\ddot r(t)$
\begin{align*}
\dot r_1(t)=r_2(t)=P_2(t), \quad
\ddot r_1(t)=r_3(t)=P_3(t)
\end{align*}
reach the highest power of $t$ of this block, where
\begin{align}\label{align proof deg P}
\deg(P_2(t))=p-2, \quad
\deg(P_3(t))=
\left\{
\begin{aligned}
&p-3, &p\geqslant 3,\\
&-\infty, &p=2.
\end{aligned}\right.
\end{align}
Here we have a convention that $\deg(0)=-\infty$.
\par
For a $2m\times 2m$ CH block with $a=\mathrm{Re}(\lambda)=0$, we have
\begin{align}\label{align CH der.}
\dot r_1(t)&=br_2(t)+r_3(t), \phantom{-}\quad
\ddot r_1(t)=-b^2 r_1(t)+2br_4(t)+r_5(t), \\
\dot r_2(t)&=-br_1(t)+r_4(t), \quad
\ddot r_2(t)=-b^2 r_2(t)-2br_3(t)+r_6(t). \nonumber
\end{align}
Note that
\begin{align}\label{align proof r}
r_1(t)&=\frac{ r_{2m-1,0} \cos{bt} + r_{2m,0} \sin{bt} }{(m-1)!} \, t^{m-1} +\sum_{\varphi=0}^{m-2}G_{\varphi}(t) t^\varphi, \\
r_2(t)&=\frac{ -r_{2m-1,0} \sin{bt} + r_{2m,0} \cos{bt} }{(m-1)!} \, t^{m-1} +\sum_{\varphi=0}^{m-2}\tilde G_{\varphi}(t) t^\varphi \nonumber
\end{align}
reach the highest power of $t$ of this block, where $G_{\varphi}(t)$ and $\tilde G_{\varphi}(t)$ are bounded functions.
By (\ref{align CH der.}), we conclude that $\dot r_1(t)$, $\ddot r_1(t)$, $\dot r_2(t)$, $\ddot r_2(t)$ can all reach the highest power $t^{m-1}$.
\par
Let $\chi$ denote the maximum value of $\nu$ in the terms of the form $\mathrm{e}^{0t} t^\nu B_\omega(t)$ in the numerator of $\kappa^2(t)$,
then by (\ref{align volume 2}), (\ref{align proof deg P}) and (\ref{align proof r}), we obtain
\begin{align}\label{align chi}
\chi\leqslant\max\{ 4(p-2), 4(m-1) \},
\end{align}
where the definitions of $p$ and $m$ are the same as (\ref{align xi}).
From (\ref{align xi}) and (\ref{align chi}), we have $\chi<\xi$, it follows that
$\lim\limits_{t\to+\infty}\kappa(t)=0$.
\end{proof}
By Lemma \ref{lemma M}, \ref{lemma 1}, and Remark \ref{rem M}, if the zero solution of the system is unstable, then $\lim\limits_{t\to+\infty}\kappa(t)=0$.
Consequently, Theorem \ref{thm Jordan}(1) is proved.

\subsection{Proof of Theorem \ref{thm Jordan}(2)}

\

Now, we prove Theorem \ref{thm Jordan}(2).
\par
By Lemma \ref{lemma M}, we only need to prove the following result.
\begin{lem}\label{lemma 2}
Under the assumptions of Theorem \ref{thm Jordan},
if $\det A \neq 0$, and $M=0$,
then $\exists t_0>0$, such that $\kappa(t)$ defined for $t\in[t_0,+\infty)$ is bounded.
\end{lem}

\begin{proof}
Assume that $\det A \neq 0$, and $M=0$.
Then $A$ has no eigenvalue $0$, and $\tilde M=M=0$.
Thus, there exist C2 or CH blocks whose eigenvalues have zero real parts in the diagonal of $A$, and
\begin{align*}
\eta\leqslant 4 M=0, \quad
\theta=6\tilde M=0.
\end{align*}

(1) If $\eta<0$, then $\lim\limits_{t\to+\infty}\kappa(t)=0$.
\par
(2) If $\eta=0$, then we compare the highest power of $t$ of terms in the form of $\mathrm{e}^{0t} t^\nu B_\omega(t)$
in the numerator and denominator of $\kappa^2(t)$, namely, $\chi$ and $\xi$,
where $B_\omega(t)$ is a bounded function.
In fact, by Remark \ref{rem sigma}, we have
\begin{align*}
\xi=6(m-1),
\end{align*}
in $V_1^6(t)$,
where $m$ denotes half of the maximum order of C2 or CH blocks with $\pm b\sqrt{-1}~(b\in\mathbb{R})$ as eigenvalues.
\par
For $V_2^2(t)$, in a C2 or CH block whose eigenvalues have zero real parts,
$\dot r_1(t)$, $\ddot r_1(t)$, $\dot r_2(t)$, $\ddot r_2(t)$ can all reach the highest power $t^{m-1}$ in the block.
\par
By (\ref{align volume 2}), we see that the maximum value $\chi$ of $\nu$ in the terms of the form $\mathrm{e}^{0t} t^\nu B_\omega(t)$ in the numerator of $\kappa^2(t)$ satisfies
\begin{align*}
\chi\leqslant 4(m-1).
\end{align*}
Therefore,
\par
(A) for $m>1$, we have $\chi<\xi$, hence $\lim\limits_{t\to+\infty}\kappa(t)=0$;
\par
(B) for $m=1$, we have $\chi=\xi=0$, thus the real Jordan blocks whose eigenvalues have zero real parts are all C2 blocks.
From (\ref{align C2 dot r}) and (\ref{align C2 ddot r}),
we see that for $i,j$ that satisfy $\mathrm{Re}(\lambda_i)=\mathrm{Re}(\lambda_j)=0$,
the function $F(t)$ in (\ref{align V2}) is bounded.
It follows that $\kappa(t)$ is a bounded function.
\par
In summary, $\exists t_0>0$, such that $\kappa(t)$ defined for $t\in[t_0,+\infty)$ is bounded.
\end{proof}

By Lemma \ref{lemma M}, \ref{lemma 2}, and Remark \ref{rem M},
if $\det A \neq 0$, and the zero solution of the system is not asymptotically stable,
then $\exists t_0>0$, such that $\kappa(t)$ defined for $t\in[t_0,+\infty)$ is bounded,
which completes the proof of Theorem \ref{thm Jordan}(2), and therefore Theorem \ref{thm Jordan} is proved.

\subsection{Proof of Theorem \ref{thm main}}\label{subsection results}

\

We prove Theorem \ref{thm main} and give several remarks in this subsection.

In what follows, we defind two subsets of $\mathbb{R}^n$ that
\begin{align*}
S&=\left\{ r(0) \Bigg| r(0)=\left(r_{10},r_{20},\cdots,r_{n0}\right)^\mathrm{T}\in\mathbb{R}^n, \mathrm{s.t.,}~\prod_{i=1}^n {r_{i0}}\neq0 \right\},
\end{align*}
and
\begin{align*}
\tilde S&=\left\{ P^{-1}v(0) \Bigg| v(0)=\left(v_{10},v_{20},\cdots,v_{n0}\right)^\mathrm{T}\in\mathbb{R}^n, \mathrm{s.t.,}~\prod_{i=1}^n {v_{i0}}\neq0 \right\}.
\end{align*}

We proved Theorem \ref{thm Jordan} in the previous subsections.
Combined with Theorem \ref{thm relation}, we have the following proposition.

\begin{prop}\label{prop v=Pr}
Take the assumptions of Theorem \ref{thm relation}, and additionally assume that $B$ is a matrix in real Jordan canonical form.
Denote by $\kappa(t)$ the first curvature of trajectory of a solution $r(t)$.
For an arbitrary initial value $r(0)\in\tilde S$, we have
\par
$(1)$ if $\lim\limits_{t\to+\infty}\kappa(t)\neq0$ or $\lim\limits_{t\to+\infty}\kappa(t)$ does not exist,
then the zero solution of the system is stable;
\par
$(2)$ if $A$ is invertible and $\lim\limits_{t\to+\infty}\kappa(t)=+\infty$, then the zero solution of the system is asymptotically stable.
\end{prop}

Noting that the Lebesgue measure of $\mathbb{R}^n \backslash \tilde S$ is zero,
we complete the proof of Theorem \ref{thm main}.

\begin{rem}
If all eigenvalues of $A$ are real numbers, we can obtain the following proposition.
\end{rem}
\begin{prop}
Under the assumptions of Theorem \ref{thm main},
together with the assumption that $A$ is invertible,
and all eigenvalues of $A$ are real numbers,
if there exists a measurable set $E\subseteq \mathbb{R}^n$ whose Lebesgue measure is greater than $0$,
such that for all $r(0)\in E$, $\lim\limits_{t\to+\infty}\kappa(t)\neq0$ or $\lim\limits_{t\to+\infty}\kappa(t)$ does not exist,
then the zero solution of the system is asymptotically stable.
\end{prop}
\begin{proof}
Under the assumptions of Theorem \ref{thm Jordan},
additionally assuming that all eigenvalues of $A$ are real numbers,
if $\det A\neq0$, and the zero solution of the system is not asymptotically stable,
then $\exists \lambda\in\sigma(A)\subseteq\mathbb{R},\ \mathrm{s.t.,}\ \lambda>0$, namely, $M>0$.
By Lemma \ref{lemma M}, we have $\lim\limits_{t\to+\infty}\kappa(t)=0$.
\par
Combined with Theorem \ref{thm relation}, we complete the proof.
\end{proof}

\begin{rem}
From the expression of $\kappa(t)$ in all cases, we see that the initial value $r(0)\in S$ does not affect the trend of the first curvature. Thus we obtain the following theorem.
\end{rem}

\begin{thm}\label{thm initial}
Under the assumptions of Theorem \ref{thm Jordan},
if for some initial value $r(0)\in S$,
we have $\lim\limits_{t\to+\infty}\kappa(t)=0$ (or $+\infty$, or a constant $C>0$, or $\kappa(t)$ is a bounded function, respectively),
then for an arbitrary $r(0)\in S$,
we still have $\lim\limits_{t\to+\infty}\kappa(t)=0$ (or $+\infty$, or a constant $\tilde{C}>0$, or $\kappa(t)$ is a bounded function, respectively).
\end{thm}

Combined with Theorem \ref{thm relation}, we have the following corollary.

\begin{cor}\label{cor initial}
Under the assumptions of Proposition \ref{prop v=Pr},
if for some initial value $r(0)\in\tilde S$,
we have $\lim\limits_{t\to+\infty}\kappa(t)=0$ (or $+\infty$, or $\kappa(t)$ is a bounded function, respectively),
then for an arbitrary $r(0)\in\tilde S$,
we still have $\lim\limits_{t\to+\infty}\kappa(t)=0$ (or $+\infty$, or $\kappa(t)$ is a bounded function, respectively).
\end{cor}

Moreover, we note that if $A$ is a matrix in real Jordan canonical form whose all eigenvalues are real numbers, then for any given $r(0)\in \mathbb{R}^n$,
we have $\lim\limits_{t\to+\infty}\kappa(t)=0$ or $C$ or $+\infty$, where $C>0$ is a constant.

\section{Examples}\label{Section Examples}

In this section, we give two examples, which correspond to each case of Theorem \ref{thm main}, respectively.

\subsection*{Example 1 (Theorem \ref{thm main}(1))}

\

Let $r(t)=\left(r_1(t),r_2(t),r_3(t),r_4(t)\right)^\mathrm{T}\in\mathbb{R}^4$, and
\begin{align*}
A=
\begin{pmatrix}
\phantom{-}10 & -20& \phantom{-}20 & -15 \\[0.7ex]
-35 & \phantom{-}20& -45 & \phantom{-}15 \\[0.7ex]
-23 & \phantom{-}26& -33 & \phantom{-}21 \\[0.7ex]
\phantom{-}36 & -32& \phantom{-}46 & -27
\end{pmatrix}.
\end{align*}
Then (\ref{system1}) becomes a four-dimensional linear time-invariant system.
If the initial value $r(0)=\left(r_{10},r_{20},r_{30},r_{40}\right)^\mathrm{T}\in\mathbb{R}^4$ satisfies
$v_{20}v_{30}v_{40}\neq0$,
where $v_{20}= -r_{10} +2r_{20} +r_{30} +3r_{40}$,
$v_{30}= r_{10} +3r_{20} +r_{30} +3r_{40}$, and
$v_{40}= r_{10} -2r_{20} +r_{30} -2r_{40}$,
then the square of the first curvature $\kappa(t)$ of curve $r(t)$ is
\begin{align*}\textstyle
\kappa^2(t)
=
\frac{5 \mathrm{e}^{40 t}
\left(
59 \mathrm{e}^{20 t} v_{20}^2 v_{30}^2
+396 \mathrm{e}^{15 t} v_{20}^2 v_{30} v_{40}
+756 \mathrm{e}^{10 t} v_{20}^2 v_{40}^2
-384 \mathrm{e}^{10 t} v_{20} v_{30}^2 v_{40}
-1188 \mathrm{e}^{5 t} v_{20} v_{30} v_{40}^2
+666 v_{30}^2 v_{40}^2
\right)}
{2 \left(
10 \mathrm{e}^{20 t} v_{20}^2
+22 \mathrm{e}^{15 t} v_{20} v_{30}
+102 \mathrm{e}^{10 t} v_{20} v_{40}
+18 \mathrm{e}^{10 t} v_{30}^2
+132 \mathrm{e}^{5 t} v_{30} v_{40}
+279 v_{40}^2
\right)^3},
\end{align*}
and we have
\begin{align*}
\lim\limits_{t\to+\infty}\kappa(t)
=\frac{\sqrt{59} \left|v_{30}\right|}{20v_{20}^2}
=\frac{\sqrt{59} \left|r_{10} +3r_{20} +r_{30} +3r_{40}\right|}
{20\left(-r_{10} +2r_{20} +r_{30} +3r_{40}\right)^2}
>0.
\end{align*}
By Theorem \ref{thm main}(1), the zero solution of the system is stable.

The graph of function $\kappa(t)$ is shown in Figure \ref{fig 1}, where $r(0)=\left(1,1,1,2\right)^\mathrm{T}$.

In another way, the eigenvalues of $A$ are $-15$, $-10$, $-5$ and $0$, thus by Proposition \ref{prop asy.stable}, the zero solution of the system is stable.

\subsection*{Example 2 (Theorem \ref{thm main}(2))}

\

Let $r(t)=\left(r_1(t),r_2(t),\cdots,r_5(t)\right)^\mathrm{T}\in\mathbb{R}^5$, and
\begin{align*}
A&=
\begin{pmatrix}
\phantom{-}0 & \phantom{-}4 & \phantom{-}5 & \phantom{-}4 & \phantom{-}1 \\[0.7ex]
-2 & -2 & \phantom{-}1 & -2 & -1 \\[0.7ex]
-2 & -4 & -3 & \phantom{-}4 & \phantom{-}3 \\[0.7ex]
\phantom{-}2 & \phantom{-}4 & \phantom{-}1 & -2 & \phantom{-}1 \\[0.7ex]
-2 & -4 & -5 & -4 & -3
\end{pmatrix}.
\end{align*}
Then (\ref{system1}) becomes a five-dimensional linear time-invariant system, and $\det A=-800\neq0$.
If $r(0)=\left(r_{10},r_{20},\cdots,r_{50}\right)^\mathrm{T}\in\mathbb{R}^5$ satisfies
$\left( r_{30}+r_{40} \right)\left( r_{40}+r_{50} \right)\neq0$,
then the square of the first curvature $\kappa(t)$ of curve $r(t)$ is
\begin{align*}
\kappa^2(t)
=
\frac{ \mathrm{e}^{4t} \left( \tilde C t^{4} + \sum_{\varphi=0}^{3} \tilde B_\varphi(t) t^{\varphi} \right) }
{ \bar C(t) t^{6} + \sum_{\psi=0}^{5} \bar B_\psi(t) t^{\psi} },
\end{align*}
where $\tilde C=102~\!400~\!C^2>0$ is a constant,
$\tilde B_\varphi(t) ~( \varphi=0,1,2,3)$,
$\bar B_\psi(t) ~( \psi=0,1,\cdots,5)$ and
$\bar C(t)=\left\{ 20~\!C\left(5-3\sin(8t+\rho)\right) \right\}^3\in\left[64~\!000~\!C^3,4~\!096~\!000~\!C^3\right]$
are bounded functions,
where $C=\left( r_{30}+r_{40} \right)^2+\left( r_{40}+r_{50}\right)^2>0$
and $\rho\in\mathbb{R}$ are constants.
Hence
\begin{align*}
\lim\limits_{t\to+\infty}\kappa(t)
=+\infty.
\end{align*}
By Theorem \ref{thm main}(2), the zero solution of the system is asymptotically stable.

The graph of function $\kappa(t)$ is shown in Figure \ref{fig 2}, where $r(0)=\left(1,1,1,1,1\right)^\mathrm{T}$.

In another way, the eigenvalues of $A$ are $\lambda_1=\lambda_2=-2+4\sqrt{-1}$, $\lambda_3=\lambda_4=-2-4\sqrt{-1}$ and $\lambda_5=-2$, thus by Proposition \ref{prop asy.stable}, the zero solution of the system is asymptotically stable.

\begin{figure}
\begin{minipage}[t]{0.5\linewidth}
\centering
\includegraphics[height=4cm]{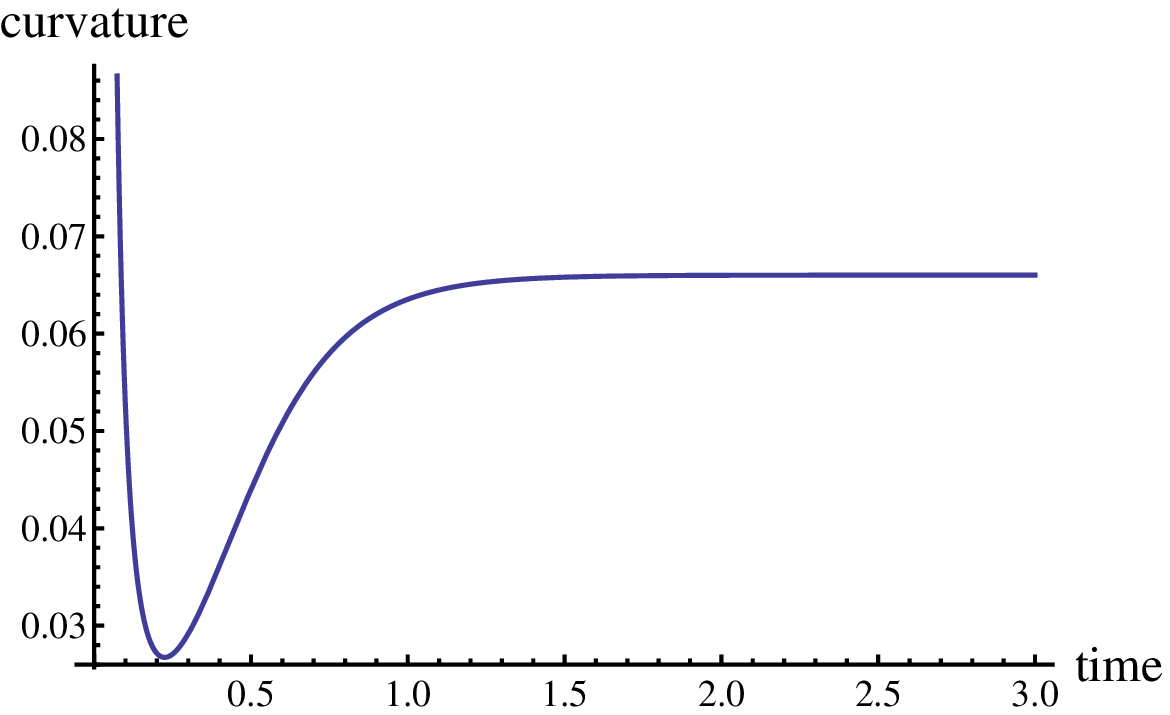}
\caption{Example 1}
\label{fig 1}
\end{minipage}%
\begin{minipage}[t]{0.5\linewidth}
\centering
\includegraphics[height=4cm]{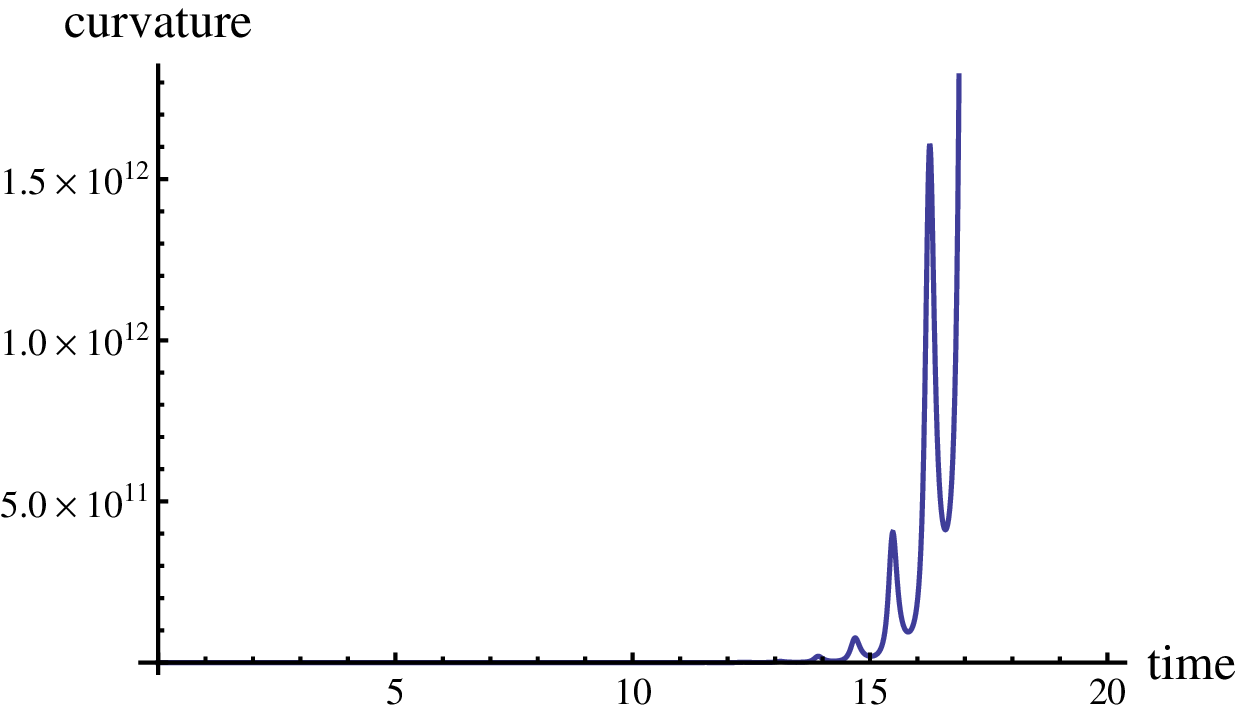}
\caption{Example 2}
\label{fig 2}
\end{minipage}
\end{figure}

\section{Conclusion}\label{Section Conclusion}

The main result of this paper, Theorem \ref{thm main}, is proved.
Firstly, through the analysis of higher curvatures of trajectories of systems,
we give the relationship between curvatures of trajectories of two equivalent linear time-invariant systems.
Secondly, for each type of real Jordan blocks,
we analyze the relationship between the first curvature and stability.
Finally, we prove a result for real Jordan canonical form, which completes the proof of the main theorem.
\par
As Theorem \ref{thm main} shows, two sufficient conditions for stability of the zero solution of linear time-invariant systems,
based on the first curvature, are given.
For each case of the theorem, we give an example to illustrate the result.
\par
Further, we will investigate nonlinear control for the stability by using geometric description.

\section*{Acknowledgment}
The research is supported partially by science and technology innovation project of Beijing Science and Technology Commission (Z161100005016043).

\end{document}